\documentclass[12pt,reqno]{amsart}
\usepackage{graphicx}
\usepackage{cases}
\usepackage{amscd}
\usepackage{amsfonts}
\usepackage{amsmath}

\usepackage{subfigure}
\usepackage{graphicx}
\usepackage{listings}
\usepackage{subfigure}
\usepackage{amsthm}
\usepackage{cite}

\theoremstyle{remark}
\newtheorem{example}{\textbf{Example}}[section]
\numberwithin{equation}{section}

\usepackage{color}
\usepackage{datetime}

\def\red{\textcolor{red}}

\allowdisplaybreaks

\usepackage[
pdfauthor={derajan},
pdftitle={How to do this},
pdfstartview=XYZ,
bookmarks=true,
colorlinks=true,
linkcolor=blue,
urlcolor=blue,
citecolor=blue,
pdftex,
bookmarks=true,
linktocpage=true, 
hyperindex=true
]{hyperref}
\usepackage{hyperref}
\allowdisplaybreaks

\usepackage{cleveref}

\makeatletter
\newcommand\figcaption{\def\@captype{figure}\caption}
\newcommand\tabcaption{\def\@captype{table}\caption}
\makeatother

\usepackage{listings}
\usepackage{geometry}
\usepackage{algorithm}
\usepackage{algorithmic}



\usepackage{multirow}

\def\bq{\begin{equation}}
\def\eq{\end{equation}}
\def\br{\begin{eqnarray}}
\def\er{\end{eqnarray}}
\def\brr{\bq\begin{array}{rlll}}
\def\err{\end{array}\eq}
\def\R{\mathbb{R}}
\def\text#1{\hbox{#1}}

\newtheorem{thm}{Theorem}[section]
\newtheorem{lem}{Lemma}[section]

\newtheorem{rem}{Remark}[section]

\newcommand{\bsub}{\begin{subequations}}
\newcommand{\esub}{\end{subequations}$\!$}

\newcommand{\argmax}{\operatornamewithlimits{argmax}}

\newcommand{\argmin}{\operatornamewithlimits{argmin}}

\title[Efficient and positive schemes for PNP systems]
{Efficient, positive,  and energy stable schemes for multi-D Poisson-Nernst-Planck systems }

\author{ Hailiang Liu and Wumaier Maimaitiyiming}
\address{Iowa State University, Mathematics Department, Ames, IA 50011}
\email{hliu@iastate.edu}\email{wumaierm@iastate.edu} 
\begin{document}

\begin{abstract} In this paper, we design, analyze, and numerically validate positive and energy-dissipating schemes for solving the time-dependent multi-dimensional system of 
Poisson-Nernst-Planck (PNP) equations, which has found much use in the modeling of biological membrane channels and semiconductor devices. The semi-implicit time discretization based on a reformulation of the system gives a well-posed elliptic system, which is shown to preserve solution positivity for arbitrary time steps. The first order (in time) fully-discrete scheme is shown to preserve solution positivity and mass conservation unconditionally, and energy dissipation with only a mild $O(1)$ time step restriction. The scheme is also shown to preserve the steady-states. For the fully second order (in both time and space) scheme with large time steps, solution positivity is restored by a local scaling limiter, which is shown to maintain the spatial accuracy. These schemes are easy to implement. Several three-dimensional numerical examples verify our theoretical findings and demonstrate the accuracy, efficiency, and robustness of the proposed schemes, as well as the fast approach to steady states.
\end{abstract}

\maketitle



\section{Introduction} 
In this paper, we are concerned with efficient and structure-preserving numerical approximations to a multi-dimensional time-dependent system of Poisson-Nernst-Planck (PNP) equations. Such system has been widely used to describe charge transport in diverse applications
 such as biological membrane channels \cite{TT53, CE93, Ei98},  electrochemical systems \cite{JN91, BTA04},
 and semiconductor devices \cite{PRS90,S84}. In the semiconductor modeling, 
 it is often called the Poisson-drift-diffusion system.  

PNP equations consist of Nernst--Planck (NP) equations that describe the drift and diffusion of ion species, 
and the Poisson equation that describes the electrostatic interaction. Such mean field approximation
 of diffusive ions  admits several variants, and we consider the following form 
\begin{subequations}\label{PNP0}
\begin{align}
\label{PNP0dr}
&  \partial_t \rho_i  + \nabla\cdot J_i=  0, \quad x\in \Omega\subset \R^d, \quad t>0, \\
\label{Flx}
& -J_i= D_i(x)\bigg[\nabla \rho_i + \frac{1}{k_BT}\rho_i(q_i\nabla \phi+\nabla \mu_i) \bigg],\\
\label{Ps}
&  -\nabla \cdot(\epsilon(x) \nabla \phi)  = 4\pi\left( f(x)+\sum_{i=1}^{m}q_i\rho_i \right),
\end{align}
\end{subequations}
subject to initial data $\rho_i(x, 0)=\rho_i^{in}(x)\geq 0$ ($i=1, \cdots, m$) and appropriate boundary conditions to be specified in section \ref{sec2.1}.  Here $m$ is the number of species,  $\rho_i=\rho_i(x,t)$ is the charge carrier density for the $i$-th species,  and $\phi=\phi(x,t)$ the electrostatic potential. The charge carrier flux is $J_i$, with which $D_i(x)$ is the diffusion coefficient,  $k_B$ is the Boltzmann constant, $T$ is the absolute temperature. The coupling parameter $q_i=z_ie$, where $z_i$ is the valence (with sign), $e$ is the unit charge. In the Poisson equation, $\epsilon(x)$ is the permittivity, $f(x)$ is the permanent (fixed) charge density  of  the system. The equations are valid in a bounded domain $\Omega$ with boundary $\partial \Omega$ and for time $t\geq  0$.
For more accurate modeling of collective interactions of charged particles, the chemical potential $\mu_i$ is often included and can be modeled by other means (see section \ref{sec2.3} for more details).  

Due to the wide variety of devices modeled by the PNP equations, computer simulation for this system of differential
equations is of great interest.
However, the PNP system is a strongly coupled system of nonlinear equations, also, the PNP system as a gradient flow can take very long time evolution to reach steady states.
Hence, designing efficient and stable methods with comprehensive numerical analysis for the PNP system is highly desirable. This is what we plan to do in this work.
\subsection{Related work}\label{sec1.1}
In the literature, there are different numerical solvers available for solving both steady and time-dependent PNP problems; see, e.g., \cite{SL01,SLL03, HEL11,GNE03, GNE04, MZLS14, ZCW11,LHMZ10}.  Many existing algorithms were introduced to handle specific issues in complex applications, in which one may encounter different numerical obstacles, such as discontinuous coefficients, singular charges, geometric singularities, and nonlinear couplings to accommodate various phenomena exhibited by biological ion channels. 
We refer the interested reader to \cite{WZCX12} for some variational multiscale models on charge transport and 
 related algorithms. 

Solutions to the PNP equations are known to satisfy some important physical properties. It is desirable to maintain these properties at the discrete level, preferably without or with only mild constraints on time step relative to spatial meshes.  Under natural boundary conditions,  three main properties for the PNP equations are known as  (i) Conservation of mass, (ii) Density positivity, and  (iii) Free energy dissipation law.
The first property requires the scheme to be conservative. The second property is point-wise and also important for the third property. In general, it is rather challenging to obtain both unconditional positivity and discrete energy decay simultaneously.  This is evidenced by several recent efforts \cite{LW14, FMEKLL13, HP16, MXL16, FKL17, LW17, GH17}, in which these properties have been partially addressed at the discrete level for PNP equations. With explicit time discretization,  the finite difference scheme in \cite{LW14} preserves solution positivity under a CFL condition $\Delta t=O(\Delta x^2)$  and the energy decay was shown for the semi-discrete scheme (time is continuous). 
An arbitrary high order DG scheme in \cite{LW17} was shown to dissipate the free energy, with solution positivity restored with the aid of a scaling limiter.  With implicit time discretization, the second order finite difference scheme in \cite{ FMEKLL13} preserves positivity under a CFL condition $\Delta t =O(\Delta x^2)$ and a constraint on spatial meshes. An energy-preserving version was further given in \cite{FKL17} with a proven second order energy decay rate. The finite element method in \cite{MXL16} employs the fully implicit backward Euler scheme to obtain solution positivity and the discrete energy decay. 
In some cases, electric energy alone can be shown to decay (see \cite{LW17}). Such decay has been verified for the finite difference scheme in \cite{HP16} and the finite element scheme in \cite{GH17}, both with semi-implicit time discretization. 

More recent attempts have focused on semi-implicit schemes based on a formulation of the nonlogarithmic Landau type.  As a result, all schemes obtained in \cite{LMCICP, LMJCAM, HPY19,DWZ19,HH19} have been shown to feature unconditional positivity ( see further discussion in section \ref{sec1.2}).  

Our goal here is to construct and analyze structure-preserving numerical schemes for  
PNP equations in a more general setting: multi-dimension, multi-species, also subject to other chemical forces. 

 \subsection{Our contributions}\label{sec1.2}  
A key step is to reformulate (\ref{PNP0dr})-(\ref{Flx}) as 
\begin{equation}\label{PNP10}
   \partial_t \rho_i = \nabla\cdot (D_i(x) e^{-\psi_i}\nabla (\rho_i e^{\psi_i})),
\end{equation}
where  
$$
\psi_i(x, t)=\frac{q_i}{k_BT}\phi(x, t)+\frac{1}{k_BT}\mu_i.
$$ 
Such reformulation, called the Slotboom transformation in the semiconductor literature, converts a drift-diffusion operator into a self-adjoint elliptic operator. It can be more efficiently solved,  and in particular more suitable for keeping the positivity-preserving property. In the context of Fokker-Planck equations it is termed as the nonlogarithmic Landau formulation (see, e.g., \cite{BD10, LY12}).  Using such reformulation in \cite{LY12} Liu and Yu constructed an implicit scheme for a singular Fokker-Planck equation and proved that all three solution properties 
hold for arbitrary time steps, for which implicit time-discretization is essential. 
Inspired by \cite{LY12, LW14}, we adopted a semi-implicit discretization of (\ref{PNP10}) in \cite{LMCICP}
to construct a first order in time and second order in space scheme for a reduced PNP system, and proved all three solution properties for the resulting scheme with only a mild $O(1)$ time step restriction. We further introduced a second order (in time) extension in \cite{LMJCAM} again for the reduced PNP system,  and a fully  second order scheme in \cite{LMJCP} for a class of nonlinear nonlocal Fokker-Planck type equations. All schemes in \cite{LMCICP,LMJCAM, LMJCP}  feature unconditional positivity  and a conditional discrete energy dissipation law simultaneously. 
 
This paper improves upon the existing results in \cite{LMCICP, LMJCAM, LMJCP} in the study of (\ref{PNP0}). We first present a semi-implicit time discretization of form 
\begin{equation}\label{PNP11}
\frac{\rho^{n+1}_i-\rho^n_i}{\tau}= \nabla\cdot (D_i(x) e^{-\psi^n_i}\nabla (e^{\psi^n_i}\rho^{n+1}_i))=:R[\rho^{n+1}_i, \psi^n_i],
\end{equation}
which is shown to be well-posed and positivity-preserving for time steps of arbitrary size and independent of the Poisson solver. We further construct the following second order scheme
 \begin{equation}\label{PNP12}
\frac{\rho^{*}_i-\rho^n_i}{\tau/2}=R[\rho^{*}_i, \frac{3}{2}\psi^n_i-\frac{1}{2}\psi^{n-1}_i], \quad \rho^{n+1}_i=2\rho^*_i-\rho^n_i,
\end{equation}
for which solution positivity for large time steps is restored by a positivity-preserving local limiter.  For the spatial discretization we use the 2nd order central difference approximation. 

Before stating the main results, let us mention some viable options in the use of reformulation (\ref{PNP10}), i.e., 
$$
\partial_t \rho_i =R[\rho_i, \psi_i],
$$
which is linear in $\rho_i$ if $\psi_i$ is a priori given.  With the second order 
central difference in spatial discretization, there are several ways to define  $\psi_i$ on cell interfaces (see section \ref{sec3.3}).  For the time discretization,  
solution positivity is readily available if we take 
\begin{equation}\label{PNPs}
\frac{\rho^{n+1}_i-\rho^{n-k+1}_i}{k\tau}=R[\rho^{n+1}_i, \psi^*_i], 
\end{equation}
with a consistent choice for $\psi_i^*$ and integer $k \geq 1$. Different options are introduced in \cite{HPY19,DWZ19,HH19} for obtaining their respective positive schemes. 

It is natural and simple to take $k=1$ and  $\psi^*=\psi^n$ in (\ref{PNPs}), 
that is (\ref{PNP11}) (again with further central difference in space).  But it is subtle to establish a discrete energy dissipation law.  
A fully discrete scheme using (\ref{PNP11}) was studied in  \cite{DWZ19}, where no energy dissipation law was established. 
Nonetheless, a discrete energy dissipation law can be verified with other options.  Indeed, (\ref{PNPs}) with $k=2$ and $\psi^*_i=\psi^n_i$ was considered in \cite{HPY19}, where the authors proved unconditional energy decay for a modified energy.
 In \cite{HH19}, (\ref{PNPs}) with $k=1$ and $\psi_i^*=(\psi^{n+1}_i+\psi^n_i)/2$ was considered, and all three properties are shown to hold simultaneously even for general boundary conditions for  the Poisson equation. 
Obviously these options can bring further computational overheads. 
 
In this work,  we formulate simple finite volume schemes for (\ref{PNP0}) by integrating 
the central difference method for spatial discretization with the semi-implicit time discretization of the reformulation (\ref{PNP10}). We have strived to advance these numerical schemes
by presenting a series of theoretical results. We summarize the main contributions as follows:
\begin{itemize} 
\item  We show that the first order time discretization gives a well-posed elliptic system (\ref{PNP11}) at each time step, and features solution positivity independent of the time steps (Theorem \ref{Semi-WELL}).  Upper bound of numerical solutions for some cases is established as well (Theorem \ref{Semi-BD}). 
\item  For the first order (in time) fully-discrete scheme, beyond the unconditional solution positivity (Theorem \ref{First-Positive}), we further establish 
a discrete energy dissipation law for time steps of size $O(1/M)$, where $M$ is the upper-bound of the numerical solutions (Theorem \ref{First-Energy}). This result sharpens the previous estimates in \cite{LMCICP} for the reduced PNP system.
We also prove that the scheme preserves steady-states, and numerical solutions converge to a steady state as $n\to \infty$ (Theorem \ref{First-Steady}).  
\item We design a fully second order (both in time and space) scheme, and solution positivity is shown for small time steps (Theorem \ref{Second-Positive}). While solution positivity for large time steps is ensured by using a local limiter.  We prove that such limiter does not destroy the 2nd order spatial accuracy (Theorem \ref{Limiter}).
\item Three-dimensional numerical tests are conducted to evaluate the scheme performance and verify our theoretical findings.  The computational cost of the second order scheme is comparable to that of the first order semi-implicit schemes (see section \ref{sec5}). 
\end{itemize}

\subsection{Organization.}\label{sec1.3}
We organize this paper as follows: In Section \ref{sec2}, we present primary problem settings and solution properties, as well as  model variations. 
In Section \ref{sec3}, we formulate a unified finite volume method for the PNP system subject to mixed boundary conditions and establish solution positivity, energy dissipation, mass conservation, and steady-state preserving properties for the case of natural boundary conditions.  Extension to a second order scheme is given in Section \ref{sec4}. In Section \ref{sec5},  we numerically verify good performance of the schemes. Finally in Section \ref{sec6} some concluding remarks are given.

Throughout this paper, we denote $\rho$ as vector $(\rho_1, \cdots, \rho_m)$,   $\partial \Omega$ 
as the boundary of domain $\Omega$ includes both the Dirichlet boundary $\partial \Omega_D$ and the Neumann boundary $\partial \Omega_N$. $|K|$ denotes the volume of domain $K$. We use $g_\alpha$ to denote 
 ${g_{\alpha}}=1/|K_{\alpha}| \int_{K_{\alpha}} g(x)dx$, for an integral average of function $g(x)$ over a cell $K_{\alpha}$. 

\section{Models and related work}\label{sec2}
\subsection{Boundary conditions} \label{sec2.1} 
 Boundary conditions are a critical component of the PNP model 
 and determine important qualitative behavior of the solution. Here we consider the simplest form of boundary conditions of Dirichlet and/or Neumann type \cite{BSW12}.
 
Let $\Omega$ be a bounded domain with Lipschitz boundary $\partial \Omega$. 
The external electrostatic potential $\phi$ is influenced by applied potential, which can be modeled by prescribing a Dirichlet boundary condition
\begin{equation}
\label{DBCP}
 \phi(x, t)=\phi^b(x, t),\quad   x\in \partial\Omega_D.
\end{equation}
For the remaining part of the boundary $\partial \Omega_N=\partial \bar \Omega \setminus \partial \Omega_D,$  a no-flux boundary  condition is applied: 
\begin{equation}
\label{NBCP}
 \epsilon(x) \nabla \phi \cdot \mathbf{n}=0,\quad x\in \partial \Omega_N.
\end{equation}
This boundary condition models surface charges, where $\mathbf{n}$ is the outward unit normal vector on the boundary $\partial \Omega_N.$  Same types of boundary conditions are imposed for $\rho_i$ as 
\begin{align}
\label{RD}
& \rho_i(x, t)=\rho_i^b(x, t)\geq 0, \quad x\in \partial \Omega_D,\\
\label{RN}
& J_i\cdot \mathbf{n}=0, \quad x\in \partial \Omega_N.
\end{align}
In this work we present our schemes by restricting to a rectangular computational domain $\Omega=(0, L_1)\times \cdots \times(0, L_d)$, with $\partial\Omega_D=\{ x\in \bar \Omega | \quad x_1=0, \ x_1=L_1\}$.

We remark that the boundary conditions for the electrostatic potential are not unique and greatly depend on the problem under investigation. For example, one 
may use a non-homogeneous Neumann boundary  condition ($
\nabla \phi \cdot \mathbf{n}=\sigma $ is used in \cite{LW17})
or Robin boundary conditions \cite{FMEKLL13, HH19}. The existence and uniqueness of the solution for the nonlinear PNP boundary value problems have been studied in \cite{JL12, LW05, PJ97} for the 1D case and in \cite{BSW12, JJ85} for multi-dimensions. 

\subsection{Positivity and energy dissipation law}\label{sec2.2}
One important solution property is 
\begin{equation}\label{positivity}
\rho_i(x, t)\geq 0,\quad x\in \Omega, \ t>0.
\end{equation}
Integration of each density equation gives 
$$
\frac{d}{dt} \int_{\Omega} \rho_i(x, t)dx =\int_{\partial \Omega} J_i\cdot \mathbf{n}ds,
$$
which with zero flux $J_i\cdot \mathbf{n}=0$ on the whole boundary leads to the mass conservation:  
\begin{equation}\label{mass}
\int_{\Omega}\rho_i(x, t)dx=\int_{\Omega}\rho_i^{in}(x)dx, \quad  t>0,\quad i=1,\cdots,m.
\end{equation}
We consider the free energy functional $E$ associated to (\ref{PNP0}) with $\mu_i=\mu_i(x)$: 
\begin{equation}\label{energy1}
E=\int_{\Omega}\bigg( \sum_{i=1}^m \rho_i(\log \rho_i-1)+\frac{1}{2k_BT}(f+\sum_{i=1}^mq_i\rho_i)\phi +\frac{1}{k_BT}\sum_{i=1}^m\rho_i\mu_i\bigg )d x.
\end{equation}
In virtue of the Poisson equation (\ref{Ps}), the free energy may be written as 
$$
E=\int_{\Omega}\bigg( \sum_{i=1}^m \rho_i(\log \rho_i-1)+\frac{\epsilon}{8\pi k_BT}|\nabla \phi|^2 +\frac{1}{k_BT}\sum_{i=1}^m\rho_i\mu_i\bigg )d x.
$$
Note that the unscaled  free energy $F=k_BT E$ is also often used, see \cite{LQL18}.   A formal calculation  gives 
\begin{align*}\label{energydiss1}
\frac{dE}{dt}& =
-\int _{\Omega} \sum_{i=1}^m D_i(x)\rho_i|\nabla \psi^*_i|^2 dx +\int_{\partial\Omega } 
\sum_{i=1}^m \psi_i^* J_i\cdot \mathbf{n} ds \\
& \quad +\frac{1}{8\pi k_BT}\int_{\partial\Omega} \epsilon(x) \left[ 
\phi (\partial_n \phi)_t - \partial_n \phi \phi_t \right] ds,
\end{align*}
where
$$
\psi^*_i:=\log \rho_i+\frac{q_i}{k_{B}T}\phi+\frac{1}{k_B T} \mu_i.
$$
Clearly, with $\partial \Omega_D=\emptyset$, we have the following energy dissipation law: 
\begin{equation}
\label{energydiss1}
\frac{dE}{dt} = -\int _{\Omega} \sum_{i=1}^m D_i(x)\rho_i|\nabla \psi^*_i|^2 dx \leq 0.
\end{equation}
Otherwise, the Dirichlet boundary condition needs to be carefully handled (see, e.g.,  \cite{LQL18}).
 
For time dependent chemical potentials $\mu_i(x,t)$,  the total free energy and its dissipation law  needs to be modified depending on how the chemical potential is determined. 

\subsection{Chemical potential}\label{sec2.3} 
In application, the chemical potential $\mu_i$ often includes the ideal chemical potential $\mu^{id}_i(x, t)$  and the excess chemical potential $\mu^{ex}_i(x,t)$ of the charged particles:
$$
\mu_i(x,t)=\mu^{id}_i(x,t)+\mu^{ex}_i(x,t),
$$
with 
$$
\mu^{id}_i(x,t)=-\log[\gamma_i\rho_i(x, t)/\rho^{bulk}_i],
$$
where the activity coefficient $\gamma_i$ described by the extended Debye-H{\"u}ckel theory depends on $\rho$ in nonlinear manner.  Meanwhile,
$$
\mu^{ex}_i(x,t)=\frac{\delta F^{ex}(\rho(x, t))}{\delta \rho_i(x, t)}
$$
is the $L^2$ variational derivative of the excess chemical functional $F^{ex}$, which may include 
hard-sphere components, short-range interactions, Coulomb interactions and electrostatic correlations, 
where the expression of each component can be found in \cite{SRZL11, MZLS14}. 

We remark that the steric interactions between ions of different species are important in the modeling of ion channels \cite{Li09, HEL11}.  Such effects can be described by choosing 
$$
F^{\rm ex}=\frac{1}{2}\int_{\Omega}\omega_{ij}\rho_i\rho_j, 
$$ 
where $\omega_{ij}$ are the second-order virial coefficients for hard spheres, depending on the size of $i$-th and $j$-th ion species \cite{ZJD17}. 
With this addition alone, the flux becomes 
$$
 -J_i= D_i(x)\left(\nabla \rho_i + \frac{1}{k_BT}q_i \rho_i \nabla \phi+\ \rho_i \sum_{j=1}^m \omega_{ij} \nabla \rho_j \right).
$$
The PNP system with this modified flux has been studied numerically first in \cite{SWZ18} without cross steric interactions, and then in \cite{DWZ19} with cross interactions. 

Our schemes will be constructed so that  numerical solutions are updated in an explicit-implicit manner while $\mu_i$ needs only to be evaluated off-line. For simplicity, we shall present our schemes assuming $\mu_i$ is given while keeping in mind that it can be applied to complex chemical potentials without difficulty. 

 \subsection{Steady states} \label{sec2.4}  By the free energy dissipation law (\ref{energydiss1}), the solution 
to (\ref{PNP0}) with zero-flux boundary conditions is expected  to converge to the steady-states as time becomes large. In such case the steady states formally satisfy (\ref{PNP0}) with $\partial_t \rho_i=0$; i.e., 
$$
\nabla\cdot ( D_i(x) \rho_i \nabla \psi_i^*)) = \nabla\cdot J_i =0, \quad  J_i\cdot \mathbf{n}=0, \quad x\in \partial \Omega.
$$
This yields $\int_{\Omega} J_i\cdot \nabla \psi_i^*dx=0$, which ensures that $\psi_i^*$ must be a constant. This gives the well-known Boltzmann distribution
\begin{equation}\label{ST1}
\rho_i=c_ie^{-\frac{1}{k_BT}(q_i\phi+\mu_i)},
\end{equation}
where $c_i$ is any constant.  Such constant can be uniquely determined by the initial data in the PNP system (\ref{PNP0}) if such steady-state is approached by the solution at large times.   Indeed, mass conservation simply gives     
\begin{equation}\label{ST2}
 c_i=\frac{\int_\Omega \rho^{in}_i dx}{\int_\Omega e^{-\frac{1}{k_BT}(q_i\phi+\mu_i)} dx}.
\end{equation}
This allows us to obtain a closed Poisson-Boltzmann equation (PBE) of form  
\begin{equation}\label{Pss}
-\nabla \cdot(\epsilon(x) \nabla \phi)  = 4\pi\left( f(x)+ \sum_{i=1}^{m}q_i c_ie^{-\frac{1}{k_BT}(q_i\phi+\mu_i)} \right), \quad \partial_{\mathbf{n}} \phi|_{\partial \Omega} =0.
\end{equation}
We should point out that the numerical method presented in this paper may be used as an iterative algorithm to 
numerically compute the nonlocal PBE (\ref{Pss}); hence it serves as a simpler alternative to the iterative DG methods recently developed  in \cite{YHL14,YHL18}.



In practical applications, one may describe ions of less interest using the Boltzmann distribution and still solve the NP equations for the target ions so to reduce the computational cost, see \cite{ZW11} for further details on related  models. Our numerical method thus provides an alternative path to simulate  such models.

\section{Numerical method}\label{sec3}
In this section we will construct positive and energy stable schemes.
\subsection{Reformulation}\label{sec3.1}
 By setting  
$$
\psi_i(x, t)=\frac{1}{k_BT}( q_i \phi(x, t)+\mu_i),
$$ 
we reformulate the density equation (\ref{PNP0dr})-(\ref{Flx}) as:
\begin{equation}\label{PNP1}
   \partial_t \rho_i = \nabla\cdot (D_i(x) e^{-\psi_i}\nabla (e^{\psi_i} \rho_i )).
\end{equation}
In spite of the aforementioned advantages of such reformulation, possible large variation of the transformed diffusion coefficients 
could result in large condition number of the stiffness matrix \cite{LHMZ10}. This issue has been recently investigated in \cite{QWZ19, DWZ19}. 

\subsection{Time discretization}\label{sec3.2}
 Let $\tau >0$ be a time step, and $t_n=\tau n$, $n=0,1\cdots,$ be the corresponding  temporal grids. 
We initialize by taking $\rho^0(x)=\rho^{in}(x)$, and obtaining $\phi^0$ by solving the Poisson equation (\ref{Ps}) using $\rho^0(x)$. 

Let $\rho^n$ and $\phi^n$ be numerical approximations of $\rho(x, t_n)$ and $\phi(x, t_n)$,  respectively, 
we first obtain 
$\rho^{n+1} $ by solving the following elliptic system:
\begin{subequations}\label{elliptic1}
\begin{align}
\label{PNPt1}
& \frac{\rho^{n+1}_i-\rho^n_i}{\tau}= \nabla\cdot (D_i(x) e^{-\psi^n_i}\nabla (e^{\psi^n_i}\rho^{n+1}_i))=:R[\rho^{n+1}_i, \psi^n_i],\\
\label{EL12}
& \rho^{n+1}_i=\rho^b_i(x,t_{n+1}) , \quad x\in \partial \Omega_D,\\
\label{EL13}
& \nabla (e^{\psi^n_i}\rho^{n+1}_i )\cdot \mathbf{n}=0, \quad x\in \partial \Omega_N,
\end{align}
\end{subequations}
where 
$$
\psi_i^n=\frac{1}{k_BT}( q_i \phi^n+\mu_i).
$$
Using this obtained $\rho^{n+1}$, we update to obtain $\phi^{n+1}$ from solving 
\begin{subequations}\label{EE2}
\begin{align}\label{Pst1}
& -\nabla \cdot(\epsilon(x) \nabla \phi^{n+1})  = 4\pi\left( f(x)+\sum_{i=1}^{m}q_i\rho^{n+1}_i \right),\\
& \phi^{n+1}(x)=\phi^b(x,t_{n+1}) , \quad x\in \partial \Omega_D,\\
\label{NFbd}
&   \nabla \phi^{n+1}\cdot \mathbf{n}=0, \quad x\in \partial \Omega_N.
\end{align}
\end{subequations}
This scheme is well-defined for any $\tau>0$ with $\rho^n \geq 0$ for all $n \in \mathbb{N}$. More precisely, we have 
\begin{thm}\label{Semi-WELL} Assume $D_i(x) \geq D_0>0$ and $\epsilon(x) \geq \epsilon_0>0$, and $\mu_i(x) \in C(\bar \Omega)$.  Then for 
given $(\rho^n, \phi^n)\in C(\bar \Omega)\cap C^2(\Omega)$, there exists a unique solution $(\rho^{n+1}, \phi^{n+1})\in C(\bar \Omega)\cap C^2(\Omega) $.  If $\rho^n\geq 0$ and $\rho^b(x,t) \geq 0$, $x\in \partial\Omega_D$, then $\rho^{n+1}\geq 0$ for $x\in \Omega$. 
\end{thm}
The proof is deferred to the appendix A.

In some cases density for the PNP problem is known to be uniformly bounded for all time. We shall show this bound property also for the semi-discrete scheme (\ref{elliptic1}). 
\begin{thm}\label{Semi-BD} Let $ 0\leq \rho_i^{in}(x)\leq B_i$, $0\leq \rho^b_i(x,t)\leq B^b_i,$
$D_i(x)/\epsilon(x)=\sigma_i$ be constants, $\Omega$ be $C^1$ convex domain, all $q_i$ have the same sign, and $\mu_i$ is smooth with $(\nabla\mu_i)\cdot\mathbf{n}\geq 0$ on $\partial\Omega_N$. 
If 
$$\tau <\frac{1}{Q_{i,max}},$$
  then $\rho^n$ obtained by scheme (\ref{elliptic1}) is uniformly bounded, i.e.,
\begin{equation}\label{result}
\rho^n_i(x) \leq \max \left\{B^b_i, \quad B_i, \quad  \frac{Q_{i,max}}{\gamma_i} \right\},
\end{equation}
where $Q_{i,max}=\max_{x\in\bar \Omega} Q_i(x)$ with 
$$
Q_i(x)=\frac{1}{k_BT} \left[ \nabla \cdot\left(D_i(x)\nabla \mu_i \right)-4\pi q_i\sigma_i f(x)   \right], 
\quad \gamma_i=\frac{4\pi q_i^2\sigma_i}{k_BT}.
$$
\end{thm}
\begin{rem} In the case of $q_i$ with different sign, density $\rho_i$ in (\ref{PNP0}) may not be bounded.  
\end{rem}
\begin{proof} 
We rewrite the semi-discrete scheme 
$$
\frac{\rho^{n+1}_i-\rho^n_i}{\tau}= \nabla\cdot \left(D_i(x) e^{-\psi^n_i}\nabla \left( \rho^{n+1}_i e^{\psi^n_i} \right) \right)
$$
into
$$
\frac{\rho^{n+1}_i-\rho^n_i}{\tau}=D_i(x)\Delta\rho^{n+1}_i+b_i \cdot \nabla\rho^{n+1}_i+c_i \rho^{n+1}_i ,
$$
where 
$$
b_i=\left( \nabla D_i(x)+D_i(x)\nabla \psi^n_i\right), \quad c_i=\nabla\cdot\left( D_i(x)\nabla\psi^n_i\right).
$$
In virtue of $\psi^n_i=\frac{q_i}{k_BT}\phi^n+\frac{1}{k_BT}\mu_i$ and $D_i(x)/\epsilon(x)=\sigma_i$,
the coefficient $c_i$ can be estimated as
\begin{equation*}\label{BD1}
\begin{aligned}
c_i=&
\frac{1}{k_BT}\left[ \nabla\cdot\left( q_i D_i(x)\nabla \phi^n \right)+\nabla \cdot\left(D_i(x)\nabla \mu_i) \right) \right]\\
=
&\frac{1}{k_BT} 
\left[ q_i\sigma_i   \nabla\cdot\left(\epsilon(x)\nabla \phi^n \right)+\nabla \cdot\left(D_i(x)\nabla \mu_i) \right) \right]\\
\qquad &  (\text{using (\ref{Pst1})})\\
=&\frac{1}{k_BT} \left[ -4\pi q_i\sigma_i  \left( f(x)+\sum_{j=1}^m q_j \rho^n_j \right)+\nabla \cdot\left(D_i(x)\nabla \mu_i) \right) \right]\\
\qquad & (\text{using $q_iq_j>0$ and $\rho_j^n \geq 0$})\\ 
\leq & \frac{1}{k_BT} \left[\nabla \cdot\left(D_i(x)\nabla \mu_i \right)  -4\pi q_i\sigma_i f(x) -4\pi q_i^2\sigma_i  \rho^n_i  \right]\\
=& Q_i(x)-\gamma_i \rho^n_i.
\end{aligned}
\end{equation*}
Hence 
\begin{equation}\label{rhoe}
\frac{\rho^{n+1}_i-\rho^n_i}{\tau}\leq D_i(x)\Delta\rho^{n+1}_i+b_i\cdot \nabla\rho^{n+1}_i+\rho^{n+1}_i \left( Q_{i,max}-\gamma_i\rho^n_i\right).
\end{equation} 
 We proceed to distinct  three cases, by letting  $x^*=\argmax_{x\in \bar\Omega}\rho^{n+1}_i(x)$:  
 
(i) If $x^*\in \partial \Omega_D$ we have
$$
\rho^{n+1}_i(x^*)=\rho^b_i(x^*, t_{n+1}) \leq B^b_i.
$$

(ii) If $x^*\in \Omega$, then (\ref{rhoe}) can be reduced to 
$$
\frac{\rho^{n+1}_i(x^*)-\rho^n_i(x^*) }{\tau}\leq \rho^{n+1}_i(x^*)\left( Q_{i,max}-\gamma_i\rho^n_i(x^*)\right).
$$
This using notation $\rho^n_{i,max}=\max_{x\in \bar\Omega}\rho^n_i$ yields 
\begin{equation}\label{BD4}
\rho^{n+1}_i(x) \leq \rho^{n+1}_i(x^*) \leq \frac{\rho^n_{i,max}}{1-\tau Q_{i,max}+\tau\gamma_i\rho^n_{i,max}}=:
P(\rho^n_{i,max}), 
\end{equation}
where we used the fact that $P(\cdot): \mathbb{R}^+ \to \mathbb{R}^+$ is non-decreasing.

(iii) If $x^*\in \partial \Omega_N$, we must have $\rho^{n+1}_i(x^*)\leq P(\rho^n_{i,max})$. 
Otherwise assume $\rho^{n+1}_i(x^*)>P(\rho^n_{i,max}).$ 
Set 
$$U(x)=\rho^{n+1}_i(x)-\rho^{n+1}_i(x^*),$$ 
and  introduce the differential operator  
$$
L\xi: =\tau D_i(x)\Delta\xi+\tau b_i \cdot \nabla\xi -(1-\tau Q_{i,max}+\tau \gamma_i \rho^n_i)\xi.
$$
From (\ref{rhoe}) we have 
$$
L \rho_i^{n+1} \geq -\rho_i^n, 
$$
and using (\ref{BD4}) we obtain  
\begin{equation*}\label{BD5}
\begin{aligned}
LU(x)= &  L\rho^{n+1}_i(x)-L\rho^{n+1}_i(x^*)\\
\geq & -\rho^n_{i} + (1-\tau Q_{i,max}+\tau \gamma_i \rho^n_{i})\rho^{n+1}_i(x^*)\\
\geq & -\rho^n_{i} + (1-\tau Q_{i,max}+\tau \gamma_i \rho^n_{i})P(\rho^n_{i,max})\\
\geq &  0.
\end{aligned}
\end{equation*}
Note that $U(x)\leq 0$ on $\partial \Omega$ and $U(x^*)=0$. Apply 
the maximum-principle  \cite[Theorem 8]{PW67} we have 
$$
(\nabla U(x^*))\cdot \mathbf{n} =(\nabla \rho^{n+1}_i(x^*))\cdot \mathbf{n}>0.
$$
On the other hand, from the no-flux boundary condition (\ref{EL13}) and using (\ref{NFbd}), 
we have 
\begin{equation*}\label{BD6}
\begin{aligned}
0=& \left(  \nabla \left(\rho^{n+1}_ie^{\psi^n_i}\right) \right)\cdot \mathbf{n}\\
= &  \left(e^{\psi^n_i} \nabla \rho^{n+1}_i +\frac{1}{k_BT}e^{\psi^n_i}\rho^{n+1}_i (q_i\nabla\phi^n+\nabla \mu_i)\right)\cdot \mathbf{n}\\
=& e^{\psi^n_i} \left(\nabla \rho^{n+1}_i\cdot \mathbf{n}+\frac{1}{k_BT}  \nabla \mu_i\cdot \mathbf{n})\right) \\
> &  e^{\psi^n_i}\frac{1}{k_BT}\nabla \mu_i\cdot \mathbf{n}  \quad x\in \partial\Omega_N.
\end{aligned}
\end{equation*}
This is a contradiction to the assumption $(\nabla\mu_i)\cdot\mathbf{n}\geq 0$.  Hence for $x\in \Omega\cup \partial\Omega_N\cup \partial \Omega_D=\bar\Omega$, we have
$$
\rho^{n+1}_{i,max}\leq \max \left\{B^b_i, \quad P(\rho^n_{i,max}) \right\}.
$$
Again by the monotonicity of $P(\cdot)$,  we obtain  
\begin{equation*}\label{BD8}
\rho^{n+1}_{i,max}\leq \max \left\{B^b_i, \quad  \max\{\rho^n_{i,max}, \quad \frac{Q_{i,max}}{\gamma_i} \} \right\}.
\end{equation*}
The stated result (\ref{result}) thus follows by induction. 
\end{proof}
A discrete energy dissipation law can be established by precisely quantifying a sufficient bound on the time step. In order to save space, we present a detailed analysis of the energy dissipation property only for the fully discrete scheme in the next section.

\subsection{Spatial discretization}\label{sec3.3}
For given positive integers $N_j$ $(j=1,\cdots,d )$, let $h_j=L_j/N_j$ be the mesh size in $j$-th direction,  $\alpha\in \mathbb{Z}^d$ be the index vector with $\alpha(j)\in \{1,\cdots,N_j\}$,  and $e_j\in \mathbb{Z}^d$ be a vector with $j$-th entry equal to one and all other entries equal to zero. We partition the domain $\Omega$ into computational cells 
$$
K_{\alpha}=[(\alpha(1)-1)h_1, \alpha(1) h_1]\times \cdots \times[(\alpha(d)-1)h_d, \alpha(d) h_d]
$$ 
with cell size $|K_{\alpha}|=\prod_{j=1}^d h_j$ such that 
$
\bigcup_{\alpha \in \mathcal{A}} K_{\alpha}=\Omega, 
$
where $\mathcal{A}$ denotes the set of all indices $\alpha$.

\subsubsection{Density update}
 A finite volume approximation of (\ref{PNPt1}) over each cell $K_{\alpha}$ with $\alpha\in \mathcal{A}$
 gives 
\begin{subequations}\label{fully1C}
\begin{equation}
 \frac{\rho^{n+1}_{i,\alpha} -\rho^n_{i,\alpha} }{\tau} =\sum_{j=1}^d \frac{C_{i,\alpha+e_j/2}-C_{i,\alpha -e_j/2}}{h_j}  = :  R_{\alpha}[\rho^{n+1}_i, \psi^n_i], 
\end{equation}
where 
$
\rho^0_{i,\alpha}: = \rho^{in}_{i, \alpha}.
$ 

Numerical fluxes on interfaces are defined by:

(i) on the interior interfaces,
\begin{align}
\label{Cfully1}
C_{i,\alpha+e_j/2} =\frac{ D_i(x_{ \alpha+e_j/2}) e^{-\psi^n_{i,\alpha+e_j/2}} }{h_j} \left(  \rho^{n+1}_{i,\alpha+e_j} e^{\psi^n_{i,\alpha+e_j}}-   \rho^{n+1}_{i,\alpha} e^{\psi^n_{i, \alpha}}      \right), \quad \text{ for }  1 < \alpha(j) < N_j;
\end{align}

(ii) on the boundary $\partial \Omega_D$, 
 \begin{equation}\label{BC11D}
\begin{aligned}
& C_{i,\alpha-e_1/2} =\frac{2D_i(x_{\alpha-e_1/2})e^{-\psi^b_i(x_{\alpha-e_1/2}, t_n )}}{h_1} \bigg(\frac{\rho^{n+1}_{i,\alpha}}{e^{-\psi^n_{i,\alpha}}} -   \frac{\rho^b_i( x_{\alpha-e_1/2}, t_{n+1})}{ e^{-\psi^b_i(x_{\alpha-e_1/2} ,t_n)}}\bigg) ,\quad \alpha(1)=1,\\
 & C_{i,\alpha+e_{1}/2} =\frac{2D_i(x_{\alpha+e_1/2}) e^{-\psi^b_i(x_{\alpha+e_1/2} , t_n )}}{h_1} \bigg(  \frac{\rho^b_i(x_{\alpha+e_1/2}, t_{n+1})}{  e^{-\psi^b_i(x_{\alpha+e_1/2} , t_n)}}- \frac{\rho^{n+1}_{i,\alpha}}{e^{-\psi^n_{i,\alpha}}}  \bigg) , \quad \alpha(1)=N_1;
 \end{aligned}
\end{equation}

(iii) on the boundary $\partial \Omega_N$,  
 \begin{equation}\label{BC11N}
\begin{aligned}
& C_{i,\alpha-e_j/2}=0,\ \text{ for  } \alpha(j)=1, \\
& C_{i, \alpha+e_j/2}=0,\ \text{ for } \alpha(j)=N_j.
 \end{aligned}
\end{equation}
\end{subequations}
In (\ref{Cfully1}) $e^{-\psi^n_{i,\alpha+e_j/2}}$  needs to be evaluated using numerical solutions $\phi^n_\alpha$.  There are three choices, all are second order approximations:

(i)  the harmonic mean
\begin{equation}\label{HM}
e^{-\psi^n_{i,\alpha+e_i/2}}=\frac{2e^{-\psi^n_{i, \alpha+e_j}-\psi^n_{i, \alpha}}}{e^{-\psi^n_{i,\alpha+e_j}}+e^{-\psi^n_{i,\alpha}}},
\end{equation}

(ii)  the geometric mean
\begin{equation}\label{GM}
e^{-\psi^n_{i,\alpha+e_i/2}}=\sqrt{e^{-\psi^n_{i, \alpha+e_j}-\psi^n_{i, \alpha}}},
\end{equation}

(iii)  the algebraic mean
\begin{equation}\label{AM}
e^{-\psi^n_{i,\alpha+e_i/2}}=\frac{e^{-\psi^n_{i, \alpha+e_j}}+e^{-\psi^n_{i, \alpha}}}{2}.
\end{equation}
 It is reported in \cite{QWZ19} that the harmonic mean results in a linear system with better condition number than that of the geometric mean. We use the harmonic mean in our numerical tests.

\subsubsection{Solving Poisson's equation}
In order to complete the scheme, we need to evaluate $\psi_{i, \alpha}^n$ by 
$$
\psi_{i, \alpha}^n=\frac{1}{k_BT}( q_i \phi_\alpha^n+\mu_{i, \alpha}),
$$
and  $\phi^n_{\alpha}$ is determined from $\rho^n_\alpha$ by using the following discretization of the equation (\ref{Pst1}):
\begin{subequations}\label{fully1Ps}
\begin{equation}
 -\sum_{j=1}^d \frac{\Phi^n_{\alpha+e_j/2}-\Phi^n_{\alpha-e_j/2}}{h_j}=4\pi\left( f_{\alpha}+\sum_{i=1}^{m}q_i\rho^n_{i,\alpha} \right),
\end{equation}
where numerical fluxes on cell interfaces are defined by:

(i) on the interior interfaces, 
\begin{align}
\label{Ps2fully1}
\Phi^n_{\alpha+e_j/2}&=\epsilon( x_{\alpha+e_j/2}) \frac{\phi^n_{\alpha+e_j}-\phi^n_{\alpha}}{h_j},\quad \text{ for } 1<\alpha(j)<N_j,
\end{align}

(ii) on the boundary $\partial \Omega_D$,  
 \begin{equation}\label{BC21D}
\begin{aligned}
\Phi^n_{\alpha-e_1/2}&=\epsilon(x_{\alpha-e_1/2}) \frac{2(\phi^n_{\alpha}-\phi^b( x_{\alpha-e_1/2}, t_n))}{h_1} , \quad \text{ for }  \alpha(1)=1,\\
\Phi^n_{\alpha+e_{1}/2}&=\epsilon(x_{\alpha+e_1/2}) \frac{2(\phi^b( x_{\alpha+e_{1}/2}, t_n)-      \phi^n_{\alpha}      )}{h_1}, \quad \text{ for }  \alpha(1)=N_1,
 \end{aligned}
\end{equation}

(iii) on the boundary $\partial \Omega_N$,   
 \begin{equation}\label{BC21N}
\begin{aligned}
& \Phi^n_{\alpha-e_j/2}=0,\ \text{ for  } \alpha(j)=1, \\
& \Phi^n_{\alpha+e_j/2}=0,\ \text{ for } \alpha(j)=N_j.
 \end{aligned}
\end{equation}
\end{subequations}
Note that in the case of $\partial\Omega_D=\emptyset$, the solution to (\ref{fully1Ps}) is unique only up to an additive constant, in such case we take $\phi^n_{(1, \cdots, 1)}=0$ to obtain a unique solution $\phi^n_{\alpha}$.

\subsubsection{Positivity}
The following theorem states that the scheme (\ref{fully1C}) preserves positivity of numerical solutions without any time step restriction.
\begin{thm} \label{First-Positive}
Let $\rho^{n+1}_{\alpha}$ be obtained from (\ref{fully1C}). If $\rho^n_{\alpha}\geq 0$ for all $\alpha\in \mathcal{A}$, and $\rho^b(x, t_n) \geq 0$, $x\in \partial\Omega_D$, then 
$$
\rho^{n+1}_{\alpha}\geq 0 \quad  \text{for all}  \quad \alpha\in \mathcal{A}.
$$
\end{thm}
\begin{proof} This proof mimics that in \cite{LY12} for the Fokker-Planck equation. Set $\lambda_j=\frac{\tau}{h_j^2}$, $\bar{D}_{i, \alpha+e_j/2}=D_i(x_{\alpha+e_j/2})e^{-\psi^n_{i,\alpha+e_j/2}}$, $g^n_{i,\alpha}=e^{\psi^n_{i,\alpha}}$ and 
$$
G_{i,\alpha}=\rho_{i,\alpha}^{n+1}g_{i,\alpha}^n ,  \quad  \alpha\in \mathcal{A}.
$$
Let $\beta$ be such that 
$$
G_{i,\beta}=\min_{\alpha\in \mathcal{A}} G_{i,\alpha},
$$
it suffices to prove $G_{i,\beta}\geq 0$. We discuss in cases:

(i) $K_{\beta}$ is an interior cell.  On the cell $K_{\beta}$ we have
\begin{align*}
g_{i,\beta}^nG_{i,\beta} & =\sum_{j=1}^d \lambda_j[\bar{D}_{i, \beta+e_j/2}(G_{i,\beta+e_j}-G_{i,\beta})-\bar{D}_{i, \beta-e_j/2}(G_{i,\beta}-G_{i,\beta-e_j})]+\rho_{i,\beta}^n \\
& \geq \rho_{i,\beta}^n, 
\end{align*}
where we used the fact $G_{i,\beta}\leq G_{i,\beta\pm e_j}$ and $\bar{D}_{i, \beta\pm e_j/2}> 0$. 
Since $g_{i,\beta}^n> 0,$ so $G_{i,\beta} \geq 0.$

(ii) $K_{\beta}$ is a boundary cell( $K_{\beta} \cap \partial \Omega_D\ne \emptyset$).  We only deal with the case $\beta (1)=1,$ remaining cases are similar.  In such case, 
 \begin{align*}
 g_{i,\beta}^nG_{i,\beta}= & \sum_{j=2}^d \lambda_j[\bar{D}_{i, \beta+e_j/2}(G_{i,\beta+e_j}-G_{i,\beta})-\bar{D}_{i, \beta-e_j/2}(G_{i,\beta}-G_{i,\beta-e_j})]\\
 & +\lambda_1 \bar{D}_{i, \beta+e_1/2}(G_{i,\beta+e_1}-G_{i,\beta})\\
 &-2\lambda_1 D_i(x_{\beta-e_1/2}) g_i^b(x_{\beta-e_1/2}, t_n) \bigg(G_{i,\beta}-   \frac{\rho^b_i(x_{\beta-e_1/2}, t_{n+1} )}{ g_i^b(x_{\beta-e_1/2}, t_n)}\bigg)+\rho_{i,\beta}^n.
 \end{align*}
 Due to $G_{i,\beta}\leq G_{i,\beta\pm e_j}$ and $\bar{D}_{i, \beta\pm e_j/2}\geq 0,$ we have
$$
\bigg( g_{i,\beta}^n+2\lambda_1 D_i(x_{\beta-e_1/2}) g_i^b( x_{\beta-e_1/2}, t_n)   \bigg)G_{i,\beta}\geq  2\lambda_1 D_i(x_{\beta-e_1/2})\rho^b_i(x_{\beta-e_1/2}, t_{n+1})+\rho_{i,\beta}^n\ge 0,
$$
which with $g_{i,\beta}^n+2\lambda_1 D_i(x_{\beta-e_1/2}) g_i^b( x_{\beta-e_1/2}, t_n)  >0$ ensures $G_{i,\beta}\geq 0.$

(iii)$K_{\beta}$ is a boundary cell ($K_{\beta} \cap \partial \Omega_N\ne \emptyset$).  Again we only deal with  
the case $\beta (l)=1$. In such case, 
 \begin{align*}
 g_{i,\beta}^nG_{i,\beta}= & \sum_{j=1, j\neq l}^d \lambda_j[\bar{D}_{i, \beta+e_j/2}(G_{i,\beta+e_j}-G_{i,\beta})-\bar{D}_{i, \beta-e_j/2}(G_{i,\beta}-G_{i,\beta-e_j})]\\
 & +\lambda_l \bar{D}_{i, \beta+1/2e_l}(G_{i,\beta+e_l}-G_{i,\beta})+\rho_{i,\beta}^n \\
 \geq &   \rho_{i,\beta}^n\ge 0. 
 \end{align*}
This also gives $G_{i,\beta}\geq 0.$  The proof is thus complete. 
\end{proof}
\subsubsection{Energy dissipation} If $\partial\Omega_{D}=\emptyset$,
then solutions $\rho^{n+1}_{\alpha}$ obtained by (\ref{fully1C}) are conservative and energy dissipating in addition to the non-negativity. Let a discrete version of the free energy (\ref{energy1}) be defined as
\begin{equation}\label{energy2}
E^n_h=\sum_{\alpha\in \mathcal{A} }|K_{\alpha}| \left[ \sum_{i=1}^m \rho^n_{i, \alpha} (\log \rho^n_{i,\alpha}-1)+\frac{1}{2k_BT}\left(f_{\alpha}+\sum_{i=1}^m q_i\rho^n_{i,\alpha} \right)\phi^n_{\alpha} +\frac{1}{k_BT}\sum_{i=1}^m\rho^n_{i,\alpha}\mu_{i,\alpha}\right ],
\end{equation} 
we have the following result.
\begin{thm}\label{First-Energy} Let $\rho_{\alpha}^n$ be obtained from (\ref{fully1C}) by using either (\ref{HM}), (\ref{GM}), or (\ref{AM}) for $e^{-\psi^n_{i,\alpha+e_i/2}}$. Let $\phi^n_{\alpha}$ be obtained from (\ref{fully1Ps}). If $\partial\Omega_{D}=\emptyset$, then we have:\\
(i) Mass conservation:
$$
\sum_{\alpha\in \mathcal{A} } |K_{\alpha}|  \rho_{i,\alpha}^{n+1}=\sum_{\alpha\in \mathcal{A} } |K_{\alpha}| \rho_{i,\alpha}^{n} \quad \text{ for } n\geq 0, i=1,\cdots, m;
$$
(ii) Energy dissipation: There exists $\tau^*>0$ such that if $\tau\in (0, \tau^*)$, then 
\begin{equation}\label{En}
E^{n+1}_h-E^n_h\leq -\frac{\tau}{2}I^n,
\end{equation}
where
$$
I^n= \sum_{i=1}^m\sum_{j=1}^d \sum_{\alpha(j) \neq N_j} |K_{\alpha}|  \frac{C_{i,\alpha+e_j/2}}{h_j} \left( \log (\rho^{n+1}_{i, \alpha+e_j}e^{\psi^n_{\alpha+e_j}})- \log (\rho^{n+1}_{i, \alpha}e^{\psi^n_{\alpha}})
\right)\geq 0.
$$
If we let 
$$
 \epsilon_{min}=\min_{x\in \bar \Omega} \epsilon(x), \ \epsilon_{max}=\max_{x\in\bar \Omega} \epsilon(x), \  D_{max}=\max_{i, x\in \bar\Omega} D_i(x), 
$$
then $\tau^*$ can be quantified by 
$$
\tau^*=\frac{ k_{B}T \epsilon^2_{min}}{4 \pi   \epsilon_{max}D_{max}\max_{i,\alpha, n} \rho^{n}_{i,\alpha} \sum_{i=1}^mq_i^2}e^{-\max_{i,j,\alpha}|\psi^n_{i,\alpha+e_j} -\psi^n_{i,\alpha}|}.
$$
\end{thm}
\begin{rem} We remark that $\tau^*$ is of size $O(1)$, though it appears to be  dependent on numerical solutions.    
For $h_j$ small, the exponential term is only of size $e^{O(h)}$, therefore bounded. As $n$ increases, the solution $\{ \rho^n_{\alpha} \}$ is expected to converge to the steady-state and therefore bounded from above, hence we simply use  the notation $\max_{i,\alpha,n}\rho^n_{i,\alpha}$. 
The boundedness of $\rho^n$ in $n$ for some cases has been established in Theorem \ref{Semi-BD} for the corresponding semi-discrete scheme. 
\end{rem}
The proof is deferred to Appendix B. 

{\color{black}
\subsubsection{Preservation of steady-states}
With no-flux boundary conditions,  scheme  (\ref{fully1C}) can be shown to be steady-state preserving.  Based on the discussion in section \ref{sec2.4},  
we say a discrete function $\rho_{\alpha}$ is at steady-state if
\begin{equation}
\rho_{i,\alpha} =c_i e^{-\frac{1}{k_BT}(q_i\phi_\alpha+\mu_{i,\alpha})}, 
 \quad i=1,\cdots,m, \quad \alpha\in \mathcal{A},
\end{equation}
where $\phi_\alpha$ satisfies (\ref{fully1Ps}) with $\rho_{i,\alpha}$ replaced by the above relation, which 
is a nonlinear algebraic equation for $\phi_\alpha$ uniquely determined for each $(c_1, \cdots c_m)$.  
We have the following theorem.
\begin{thm}\label{First-Steady}
 Let the assumptions in Theorem \ref{First-Energy} be met, then

(i) If $\rho^{0}_\alpha$ is already at steady-state, then $\rho^n_\alpha=\rho^{0}_\alpha$ for $n\geq 1$.

(ii) If $E^{n+1}_h=E^n_h$, then $\rho^n_\alpha$ must be at steady-state.

(iii) If  $\rho^n_{i,\alpha}, \phi^n_{\alpha}$ converge as $n\to \infty$, then  their limits are determined by   
$$
\rho^\infty_{i,\alpha}=c^\infty_ie^{-\frac{1}{k_BT}(q_i\phi^\infty_\alpha+\mu_{i,\alpha})}, \quad c^\infty_i=\frac{\sum_{\alpha\in \mathcal{A}} |K_\alpha| \rho^{0}_{i,\alpha} }{\sum_{\alpha\in \mathcal{A}} |K_\alpha| e^{-\frac{1}{k_BT}(q_i\phi^\infty_\alpha+\mu_{i,\alpha})} },
$$
where $\phi^\infty_\alpha$ is obtained by solving (\ref{fully1Ps}) by using $\rho^\infty_{i,\alpha}$.
\end{thm}

\begin{proof}
(i) We only need to prove $\rho^1_{i,\alpha}=\rho^0_{i,\alpha},$  for all $ i=1,\cdots,m, \ \alpha\in \mathcal{A}$.
Summing (\ref{fully1C}) with $n=0$ against $|K_\alpha| \rho^1_{i,\alpha}/\rho^0_{i,\alpha}$, using summation by parts, we obtain  
\begin{equation}\label{ST11}
\begin{aligned}
\sum_{\alpha\in \mathcal{A}}|K_\alpha| (\rho^{1}_{i,\alpha} -\rho^0_{i,\alpha})\frac{\rho^1_{i,\alpha}}{\rho^0_{i,\alpha}}=   & \tau \sum_{j=1}^d \sum_{\alpha\in \mathcal{A}} |K_\alpha| \frac{1}{h_j} (C_{i,\alpha+e_j/2}-C_{i,\alpha -e_j/2}) \frac{\rho^1_{i,\alpha}}{\rho^0_{i,\alpha}}\\
=& - \tau \sum_{j=1}^d \sum_{\alpha(j)\neq N_j} |K_\alpha| \frac{1}{h_j} C_{i,\alpha+e_j/2} \left( \frac{\rho^1_{i,\alpha+e_j}}{\rho^0_{i,\alpha+e_j}} - \frac{\rho^1_{i,\alpha}}{\rho^0_{i,\alpha}}  \right).
\end{aligned}
\end{equation}
Substituting  $\rho^0_{i,\alpha} =c_i e^{-\psi^0_{i, \alpha}}$ into $C_{i,\alpha+e_j/2}$, the right hand side of 
(\ref{ST11}) becomes  
\begin{align*}
RHS=  & - \tau  c_i\sum_{j=1}^d \sum_{\alpha(j)\neq N_j} |K_\alpha| \frac{ D_{i, \alpha+e_j/2}  e^{-\psi^0_{i,\alpha+e_j/2}} }{h^2_j}  \left( \frac{\rho^1_{i,\alpha+e_j}}{\rho^0_{i,\alpha+e_j}} - \frac{\rho^1_{i,\alpha}}{\rho^0_{i,\alpha}}  \right)^2\leq 0.
\end{align*}
Adding 
$
\sum_{\alpha\in \mathcal{A}}|K_\alpha| (\rho^{0}_{i,\alpha} -\rho^1_{i,\alpha})=0
$
 to the left hand side of (\ref{ST11}) leads to
\begin{align*}
LHS=  & \sum_{\alpha\in \mathcal{A}}|K_\alpha| \left [(\rho^{1}_{i,\alpha} -\rho^0_{i,\alpha})\frac{\rho^1_{i,\alpha}}{\rho^0_{i,\alpha}} +(\rho^{0}_{i,\alpha} -\rho^1_{i,\alpha})\right]\\
=& \sum_{\alpha\in \mathcal{A}}|K_\alpha|   \frac{ (\rho^{1}_{i,\alpha} -\rho^0_{i,\alpha})^2}{\rho^0_{i,\alpha}} \geq 0.
\end{align*}
Hence $LHS=RHS \equiv 0$,  we must have 
$$
\rho^1_{i,\alpha}=\rho^0_{i,\alpha}, \quad i=1,\cdots,m, \quad \alpha\in \mathcal{A}.
$$

(ii) The inequality (\ref{En}) when combined with  $E^{n+1}_h=E^n_h$  leads to $I^n=0$.  
From the proof  of Theorem \ref{First-Energy} in Appendix B it follows 
$$
\rho^{n+1}_{i,\alpha}=\rho^n_{i,\alpha}.
$$

(iii) Since $E^n_h$ is non-increasing in $n$, and we can verify that $E_h^n$ is bounded from below,  
hence 
$$
\lim_{n\to \infty} E^n_h=\inf \{E^n_h\}.
$$
Taking the limit in (\ref{En}), we have $\lim_{n\to \infty} I^n=0$, which implies  
$$
\rho^{\infty}_{i,\alpha} =c^\infty_i e^{-\psi^\infty_{i, \alpha}}. 
$$
Conservation of mass gives 
$$
c^\infty_i=\frac{\sum_{\alpha\in \mathcal{A}} |K_\alpha| \rho^{0}_{i,\alpha} }{\sum_{\alpha\in \mathcal{A}} |K_\alpha| e^{-\psi^\infty_{i,\alpha}}}.\quad i=1,\cdots,m, \quad \alpha\in \mathcal{A},
$$
where $\phi^\infty_\alpha$ in $\psi^\infty_{i, \alpha} = \frac{1}{k_BT}(q_i\phi^\infty_\alpha+\mu_{i,\alpha}) $ is obtained by solving (\ref{fully1Ps}) using $\rho^\infty_{i,\alpha}$.
\end{proof}

}

\section{Second order in time discretization} \label{sec4}
The semi-discrete scheme (\ref{PNPt1}) is first order accurate, one can design higher order in time schemes
based on (\ref{PNP1}).
 
The following is a second order time discretization, 
$$
\frac{\rho^{n+1}_i-\rho^n_i}{\tau}=R[(\rho^{n+1}_i+\rho^{n}_i)/2, \frac{3}{2}\psi^n_i-\frac{1}{2}\psi^{n-1}_i].
$$
This can be expressed as a prediction-correction method, 
\begin{equation}\label{PNPt2}
\frac{\rho^{*}_i-\rho^n_i}{\tau/2}=R[\rho^{*}_i, \frac{3}{2}\psi^n_i-\frac{1}{2}\psi^{n-1}_i], \quad \rho^{n+1}_i=2\rho^*_i-\rho^n_i.
\end{equation}
As argued for the first order scheme, this scheme is well-defined. 
\subsection{Second order fully-discrete scheme}\label{sec4.1}
With central spatial difference, our fully discrete second order (in both space and time) scheme reads
\begin{subequations}\label{fully2Ci}
\begin{align}
\label{PNPtp}
& \frac{\rho^{*}_{i, \alpha}-\rho^n_{i, \alpha}}{\tau/2}= R_{\alpha}[\rho^*_i, \frac{3}{2}\psi^n_i-\frac{1}{2}\psi^{n-1}_i],\\
\label{PNPtc}
& \rho^{n+1}_{i, \alpha}=2\rho^*_{i, \alpha}-\rho^n_{i,\alpha}.
\end{align}
\end{subequations}
 Positivity of $\rho^{n+1}_{\alpha}$ can be ensured if time steps are sufficient small. 
\begin{thm}\label{Second-Positive}
 Let  $\rho^{n+1}_{\alpha}$ be obtained from (\ref{fully2Ci}). If $\rho^n_{\alpha}\geq 0$ for all $\alpha\in \mathcal{A}$, and $\rho^b(x,t) \geq 0$ for $x\in \partial \Omega_D$, then 
$$
\rho^{n+1}_{\alpha}\geq 0, \quad \alpha\in \mathcal{A}
$$ 
provided $\tau$ is sufficiently small.
\end{thm}
\begin{proof}
Inserting (\ref{PNPtc}) into (\ref{PNPtp}) leads to the following compact  form of the scheme (\ref{fully2Ci}):
\begin{equation}\label{comp}
\rho^{n+1}_{i,\alpha}-\frac{\tau}{2}R_{\alpha}[\rho^{n+1}_i, \frac{3}{2}\psi^n_i-\frac{1}{2}\psi^{n-1}_i]
=\rho^n_{i,\alpha}+\frac{\tau}{2}R_{\alpha}[\rho^n_i, \frac{3}{2}\psi^n_i-\frac{1}{2}\psi^{n-1}_i],
\end{equation}
where we have used the linearity of $R_{\alpha}[\cdot, \cdot]$ on the first entry.

Set 
$$
g^*_{i,\alpha}=e^{\frac{3}{2}\psi^n_{i,\alpha}-\frac{1}{2}\psi^{n-1}_{i,\alpha}}, \  \bar D^*_{i,\alpha+e_j/2}=D_{i,\alpha+e_j/2}e^{-\frac{3}{2}\psi^n_{i,\alpha+e_j/2}+\frac{1}{2}\psi^{n-1}_{i,\alpha+e_j/2}},\ 
G^{n}_{i,\alpha}=\rho^n_{i,\alpha} g^*_{i,\alpha},
$$
then the scheme (\ref{comp}) can be rewritten as
\begin{equation}\label{p21}
\begin{aligned}
& g^*_{i,\alpha}G^{n+1}_{i,\alpha}-\sum_{j=1}^d \frac{\tau}{h_j^2}[\bar{D}^*_{i, \alpha+e_j/2}(G^{n+1}_{i,\alpha+e_j}-G^{n+1}_{i,\alpha})-\bar{D}^*_{i, \alpha-e_j/2}(G^{n+1}_{i,\alpha}-G^{n+1}_{i,\alpha-e_j})]\\
& =g^*_{i,\alpha}G^n_{i,\alpha}+\sum_{j=1}^d \frac{\tau}{h_j^2}[\bar{D}^*_{i, \alpha+e_j/2}(G^{n}_{i,\alpha+e_j}-G^{n}_{i,\alpha})-\bar{D}^*_{i, \alpha-e_j/2}(G^{n}_{i,\alpha}-G^{n}_{i,\alpha-e_j})].
\end{aligned}
\end{equation}
Let $\beta$ be such that 
$$
G^{n+1}_{i,\beta}=\min_{\alpha\in \mathcal{A}} G^{n+1}_{i,\alpha},
$$
 it suffices to prove $G^{n+1}_{i,\beta}\geq 0$. We prove the result when $K_{\beta}$ is an interior cell, the result for boundary cells can be proved similarly. 

Since $G^{n+1}_{i,\beta}\leq G^{n+1}_{i,\beta\pm e_j}$ and $G^n_{i,\beta\pm j}\geq 0$, thus equation (\ref{p21}) on cell $K_{\beta}$ reduces to the inequality:
$$
g^*_{i,\beta}G^{n+1}_{i,\beta}\geq \left(g^*_{i,\beta}- \tau \sum_{j=1}^d \frac{1}{h_j^2}(\bar{D}^*_{i, \beta+e_j/2}+\bar{D}^*_{i, \beta-e_j/2})\right)G^n_{i,\beta},
$$
we see that $G^{n+1}_{i,\beta}\geq 0$ is insured if 
$$
\tau\leq \min_{\alpha} \left \{ \frac{g^*_{i,\alpha}} { \sum_{j=1}^d \frac{1}{h_j^2}(\bar{D}^*_{i, \alpha+e_j/2}+\bar{D}^*_{i, \alpha-e_j/2}) } \right \}.
$$
The stated result thus follows.
\end{proof}

We should point out that numerical density $\{ \rho^n_{\alpha}\}$ obtained by the second order scheme (\ref{fully2Ci}) may not be non-negative for large time step $\tau$, though $\{ \rho^*_{\alpha}\}$ stays positive.  
We shall restore solution positivity by using a local limiter, which 
was first introduced in \cite{LMJCAM} for one-dimensional case.

\subsection{Positivity-preserving limiter}\label{sec4.2}
  We present a local limiter to restore positivity of  $\rho$ if  
$$
\sum_{\alpha \in \mathcal{A} }|K_\alpha| \rho_\alpha >0,
$$ 
but $\rho_\beta<0$ for some $\beta \in \mathcal{A}$.  The idea is to find a neighboring index set $ S_\beta$ such that the local average 
$$
\bar{\rho}_\beta=\frac{1}{|S_{\beta}|}\sum_{\gamma \in S_{\beta}} |K_\gamma|\rho_{\gamma} >0,
$$
where $|S_{\beta}|$ denotes the minimum number of indices for which $\rho_{\gamma} \ne 0$ and $\bar \rho_\beta>0$, then use this local average as a reference to define the following scaling limiter
\begin{equation}\label{lim}
\tilde{\rho}_{\alpha}=\theta \rho_{\alpha}+(1-\theta)\bar{\rho}_\beta/|K_\alpha|, \quad \alpha\in S_{\beta},
\end{equation}
where 
$$
\theta=\min \left\{1,  \frac{\bar{\rho_\beta}}{\bar{\rho}_\beta -\rho_{\min}}\right\}, 
\quad \rho_{\min}=\min_{\gamma\in S_{\beta}} |K_\gamma|\rho_{\gamma}.
$$
Recall the result stated in Lemma 5.1 in \cite{LMJCP}, such limiter restores solution positivity and respects the local mass conservation. In addition, for any sequence $g_\alpha$ with $g_\alpha \geq 0$, 
 we have 
\begin{align}\label{pg}
|\tilde \rho_\alpha -g_\alpha |\leq (1+|S_\beta| \Lambda)\max_{\gamma  \in S_\beta}|\rho_\gamma -g_\gamma|, \quad\alpha \in S_\beta,
\end{align}
where $\Lambda$ is the upper bound of mesh ratio $|K_\gamma|/|K_\alpha|$.  Let $\rho_\alpha$ be the approximation of $\rho(x)\geq 0$, we let $g_\alpha=\rho(x_\alpha)$ or the average of $\rho$ on $K_\alpha$, 
so we can assert that the accuracy is not destroyed by the limiter as long as $|S_\beta|\Lambda$ is uniformly bounded. Boundedness of $|S_{\beta}|$ for shape-regular meshes was rigorously proved in \cite{LMJCP} 
for the one-dimensional case. We restate such result in the present setting in the following.
\begin{thm}\label{Limiter} 
Let $\{\rho_\alpha\}$ be an approximation of $\rho(x) \geq 0$ over shape regular meshes, and $\rho \in C^k(\Omega)$ ($k\geq 2$).  If $\rho_\beta<0$ (or only finite number of neighboring values are negative), then there exists $C^*>0$ finite such that 
$$
|\tilde \rho_\alpha -\rho(x_\alpha)|\leq C^*\max_{\alpha \in S_\beta} |\rho_\alpha -\rho(x_\alpha)|, \quad 
\forall \alpha \in S_\beta,
$$ 
where $C^*$ may depend on the local meshes associated with $S_\beta$.  
\end{thm}
\begin{proof}
For simplicity, we prove only for the case of uniform meshes (e.g. uniform in each dimension).
Let $h=\min_{1\leq j \leq d}h_j \leq 1$ and $h_j\leq \Lambda h$ for some $\Lambda>0$. From (\ref{pg}) we see that 
it suffices to show there exists $A^*>0$ finite such that 
$
|S_\beta|\leq A^*, 
$
with which we will have $C^*=1+A^*\Lambda$.  Under the smoothness assumption of $\rho$ we may assume $|\rho_\alpha-\rho(x_\alpha)|\leq C h^k$. 
Under the assumption $\rho_{\beta}<0$, $\rho$ must touch zero near $x_{\beta}$. We discuss the case  where $\rho(x^*)=0$ and $\nabla \rho(x^*)=\vec 0$ with $\rho(x)>0$ for $x(j) \geq x^*(j)$, $j=1,\cdots, d,$ locally with $x^*\in K_{\beta}.$ To be concrete, we consider $\beta=(1,\cdots,1)$ and $\int_{K_{\beta}}\rho(x)dx>0$. From the limiter construction we have $S_\beta$ such that 
\begin{equation}\label{Lim1}
\sum_{\alpha\in S_{\beta}}|K_\alpha|\rho_\alpha>0.
\end{equation}
The rest of the proof is devoted to bounding $|S_\beta|$.  The assumed error bound gives  
\begin{equation}\label{Lim2}
\rho_{\alpha}\geq \rho(x_{\alpha})-Ch^k.
\end{equation}
From $\rho\in C^k(\Omega) (k\geq 2)$, we have 
\begin{equation}\label{Lim3}
\rho(x_{\alpha})\geq \bar\rho_\alpha-\lambda \Lambda^2 h^2,
\end{equation}
with $\lambda=\frac{d}{24}\max_{j=1,\cdots,d} |\partial_{x_jx_j}\rho|$ and the cell average $\bar\rho_\alpha=\frac{1}{|K_{\alpha}|}\int_{K_\alpha}\rho(x)dx$. From (\ref{Lim2}) and (\ref{Lim3}), we see that the left hand side of (\ref{Lim1}) is bounded from below by  
\begin{equation}\label{Lim4}
\begin{aligned}
\sum_{\alpha\in S_{\beta}}|K_\alpha|\rho_\alpha \geq &\sum_{\alpha\in S_{\beta}} |K_\alpha| \left  (\bar\rho_\alpha-(C+\lambda \Lambda^2) h^2 \right)\\
= & \int_{\cup_{\alpha\in S_\beta} K_\alpha}\rho(x)dx-(C+\lambda \Lambda^2) h^2 \sum_{\alpha\in S_{\beta}} |K_\alpha|.
\end{aligned}
\end{equation}
Without loss of generality we assume $\cup_{\alpha\in S_\beta} K_\alpha$ is a rectangle in $\R^d$;  otherwise we could add more cells to complete the rectangle.
Let 
$$
|\cup_{\alpha\in S_\beta} K_\alpha|=\Pi_{j=1}^d \eta_j, \quad h_j\leq  \eta_j \leq L_j, 
$$
and  $\vec \eta=(\eta_1, \cdots, \eta_d), \ \vec h=(h_1, \cdots, h_d).$
Rewriting integral  in (\ref{Lim4}) we have  
$$
\sum_{\alpha\in S_{\beta}}|K_\alpha|\rho_\alpha \geq \left[ g(\eta) - (C+\lambda \Lambda^2) h^2\right] \sum_{\alpha\in S_{\beta}} |K_\alpha|,
$$
where 
$$
g(\eta):=\int_0^1\!\cdots \!\int_0^1\!\! \rho\left( {\rm diag}(\vec  \theta)\vec \eta+x_\beta-\frac{1}{2}\vec h \right) d\theta_1 \cdots d\theta_d.
$$
From the fact
$
h^d\leq \frac{\eta_1\cdots \eta_d}{|S_\beta|},
$
we can see that the term in the bracket is bounded from below by
$$
g(\eta) -(C+\lambda \Lambda^2)  \left(\frac{\eta_1\cdots \eta_d}{|S_\beta|}\right)^{2/d},
$$
which is positive if 
$$
|S_\beta|> (C+\lambda \Lambda^2)^{d/2} {g(\eta)}^{-d/2} \eta_1\cdots \eta_d. 
$$
This can be insured if we take 
$$
|S_\beta|=\lfloor A \rfloor+1, 
$$
where 
$$
A=(C+\lambda\Lambda^2)^{d/2} \max_{\eta_j\in [h_j, L_j], j=1,\cdots, d} {g(\eta)}^{-d/2} \eta_1\cdots \eta_d. 
$$
This is bounded and may depend on the local mesh of $K_\beta$. 
\end{proof}

Note that our numerical solutions feature the following property: if $\rho^n_{i, \alpha}= 0$, then 
$$
\rho^{n+1}_{i, \alpha} =2\rho^*_{i, \alpha} - \rho^n_{i, \alpha} \geq 0
$$
due to the fact $\rho^*_{i, \alpha} \geq 0 $. This means that if $\rho^{\rm in}(x)=0$ on an interval, 
then $\rho^1_{i, \alpha}$ cannot be negative in most of nearby cells. Thus negative values appear only where 
the exact solution turns from zero to a positive value, and the number of these values are finitely many. 
Our result in Theorem \ref{Limiter} is thus applicable.

\subsection{Algorithm}\label{sec4.3}
The following algorithm is only for the second order scheme with limiter.
\begin{enumerate}
\item Initialization: From the initial data $\rho^{in}_i(x)$, obtain 
$$\rho^0_{i,\alpha}=\frac{1}{|K_{\alpha}|}\int_{K_{\alpha}}\rho_i^{in}(x)dx,\  i=1,\cdots,m,\ \alpha\in \mathcal{A},
$$ 
by using central point quadrature.
\item Update to get $\{\rho^1_{i,\alpha}\}$:  Compute $\{\phi^0_{\alpha}\}$ from (\ref{fully1Ps}), 
 then obtain $\{\rho_{i,\alpha}^1\}$ by the first order scheme (\ref{fully1C}).
\item Update from $\{\rho^n_{i,\alpha}\}$: For $n\geq 1$, compute $\{\phi_{\alpha}^n\}$ from scheme  (\ref{fully1Ps}) 
then get $\{\rho_{i,\alpha}^{n+1}\}$ from (\ref{fully2Ci}).
\item Reconstruction: if necessary, locally replace $\rho_{i,\alpha}^{n+1}$ by $\tilde \rho_{i,\alpha}^{n+1}$ using the limiter defined in (\ref{lim}). 
\end{enumerate}
The following algorithm can be called to find an admissible set $S_{\alpha}$  used in (\ref{lim}).
\begin{enumerate}
\item[(i)] Start with $S_{\beta}=\{ \beta\}$, $p=1$.
\item[(ii)] For $l_j=\max \{1, \alpha (j)-p\}: \min \{\alpha (j)+p, N_j\}$ with $j=1,\cdots, d$. 
\\
 If $\alpha:=(l_1, \cdots, l_d)\notin S_{\beta}$ and $\rho^{n+1}_{i, \alpha}\ne 0$, 
 then set $S_{\beta}=S_{\beta}\cup \left\{\alpha \right\}$.\\
 If $\bar{\rho}_\beta >0$, then stop, else go to (iii).
\item[(iii)] Set $p=p+1$ and go to (ii).
\end{enumerate}
\begin{rem}
The coefficient matrices of the linear systems obtained by (\ref{fully1C}), (\ref{fully1Ps}), and  (\ref{PNPtp}) are sparse, diagonally dominant, and symmetric, hence
more efficient linear system solvers, such as the ILU preconditioner + FGMRES (see
e.g., \cite{S93}), ILU preconditioner + Bicgstab (see e.g., \cite{CMZ16}), can be used.
\end{rem}

%
%

\section{Numerical tests}\label{sec5}

In this section, we implement the fully discrete schemes (\ref{fully1C}) and (\ref{fully2Ci}) 
to demonstrate their orders of convergence and capacity to preserve solution properties. 
In both schemes the numerical solution  $\phi^n_\alpha$ is computed by the scheme (\ref{fully1Ps}). 
Errors in the accuracy tests are measured in the following discrete $l^1$ norm:
$$
\text{error}=\sum_{\alpha\in \mathcal{A}}|K_\alpha||\tilde g_{\alpha}- g_{\alpha}|.
$$
Here $ g_{\alpha}$ denotes the numerical solution, say $g_{\alpha}=\rho_{i,\alpha}^n$ or $\phi_{\alpha}^n$ at time $t=n\tau$, and $ \tilde g_{\alpha}$ indicates the cell average of the corresponding 
exact solutions. 

In our numerical tests, the sparse linear systems obtained by (\ref{fully1C}), (\ref{fully1Ps}), and  (\ref{PNPtp}) are solved by  ILU preconditioned  FGMRES \cite{S93} algorithm using compressed row format of the coefficient matrices. In the three-dimensional case, the coefficient matrices of the linear systems are 7-diagonal matrices. It is worth to mention that the compressed row format allows us to store a 
$l\times l$ 7-diagonal matrix by using at most $15l$ storage locations with $l=N_x\times N_y\times N_z$.  With $30\times30\times30$ cells, we can save $99\%$ of the storage space needed for storing the resulting coefficient matrices.

In our three examples below we consider the computational domain
$$
\Omega=(0,1)\times(0,1)\times(0,1). 
$$

\begin{example}\label{ex1}(Accuracy test) 
In this test we numerically verify the accuracy and order of schemes (\ref{fully1C}) and  (\ref{fully2Ci}) by using manufactured solutions.  Consider   
  \begin{equation}
 \left \{
\begin{array}{rl}
 \hfill    \rho_1(\mathbf{x},t)& =4(x^2(1-x)^2+y(1-y))e^{-t},\\
 \hfill    \rho_2(\mathbf{x},t) & =(y(1-y)+z^2(1-z)^2)e^{-t},\\
 \hfill    \phi(\mathbf{x},t) &=(x^2(1-x)^2+y(1-y)+z^2(1-z)^2)e^{-t}
\end{array}
\right.
\end{equation}
and
$$
\partial\Omega_D=\{\mathbf{x}\in\bar  \Omega : y=0,1\}, \quad \partial \Omega_N=\partial \bar \Omega \setminus \partial \Omega_D,
$$
then they are exact solutions to the following problem
  \begin{equation}
 \left \{
\begin{array}{rl}
 \hfill  \partial_t \rho_1 =&\nabla \cdot(\nabla \rho_1 +  \rho_1 \nabla \phi)+f_1(\mathbf{x},t),  \hfill \ \ \  \mathbf{x} \in \Omega, \ t>0,\\
 \hfill  \partial_t \rho_2 =&\nabla \cdot(\nabla \rho_2 -  \rho_2 \nabla \phi)+f_2(\mathbf{x},t),  \hfill \ \ \   \mathbf{x} \in \Omega, \ t>0,\\
 \hfill  -\Delta \psi =&\rho_1 -   \rho_2 +f_3(\mathbf{x},t),  \hfill \ \ \   \mathbf{x} \in \Omega, \ t>0,
\end{array}
\right.
\end{equation}
where source terms $f_1(\mathbf{x},t), f_2(\mathbf{x},t)$ and $f_3(\mathbf{x},t),$ and the initial and boundary conditions  are determined by the exact solutions.

We first test the accuracy of the semi-implicit scheme (\ref{fully1C}) by using various spatial step size $h$, 
errors and orders at $t=1$ are listed in Table 1 (with  $\tau=h$) and in Table 2 (with  $\tau=h^2$), respectively. 
We observe the first order accuracy in time and the second order accuracy in space. We then test the accuracy of  the scheme (\ref{fully2Ci}) with time step size $\tau=h$. From Table 3, we  see the second order accuracy in both time and space.  
 \small{
  \begin{table}[ht]\label{ex11}
        \centering
                \caption{\tiny{Scheme  (\ref{fully1C}) with $\tau=h$ }}
        \begin{tabular}{|c| c |c |c| c |c|c| c |}
            \hline
              $N_x\times N_y\times N_z$ & \  $\rho_1$ error & order&  $\rho_2$ error &order & $\phi$ error  &order \\ [0.5ex] 
            \hline
           $8\times 8\times 8$&   4.7508E-02&                  -&  1.3904E-02&           - &5.7213E-03&-\\
            $16\times 16\times 16$&  2.1283E-02&  1.1585 &  5.8701E-03&  1.2440 &2.0987E-03 & 1.4468\\
            $32\times 32\times 32$ &  1.0060E-02&  1.0811 &   2.6956E-03& 1.1228   & 8.6460E-04& 1.2794\\
            $64\times 64\times 64$&   4.8890E-03&  1.0410 &   1.2915E-03& 1.0616  &3.8667E-04& 1.1609
             \\ [1ex]
            \hline
        \end{tabular}
     \end{table}
  \begin{table}[ht]\label{ex12}
              \centering
                \caption{\tiny{Scheme  (\ref{fully1C})  with $\tau=h^2$ }}
        \begin{tabular}{|c| c |c |c| c |c|c| c |}
            \hline
              $N_x\times N_y\times N_z$ & $\rho_1$ error & order& $\rho_2$ error &order & $\phi$  error & order \\ [0.5ex] 
            \hline
          $8\times 8\times 8$&   1.1252E-02&                  -&  4.0301E-03&-        &3.1194E-03&-\\
            $16\times 16\times 16$&  2.7824E-03&  2.0158 &  9.8548E-04&  2.0319  &7.7117E-04&2.0161 \\
            $32\times 32\times 32$ &  6.9369E-04&  2.0040 &   2.4502E-04& 2.0079   & 1.9225E-04& 2.0041 \\
            $64\times 64\times 64$&   1.7330E-04&  2.0010&   6.1170E-05&  2.0020  &4.8028E-05& 2.0010
             \\ [1ex]
            \hline
        \end{tabular}  
          \end{table}
  \begin{table}[ht]\label{ex13}
        \centering
                \caption{\tiny{Scheme (\ref{fully2Ci}) with $\tau=h$ }}
        \begin{tabular}{|c| c |c |c| c |c|c| c |}
            \hline
              $N_x\times N_y\times N_z$ &  $\rho_1$ error & order& $\rho_2$ error &order & $\phi$  error &order \\ [0.5ex] 
            \hline
          $8\times 8\times 8$&   5.5476E-03&               -&  2.3247E-03&             - &2.7378E-03&-\\
            $16\times 16\times 16$&  1.5073E-03& 1.8799&  6.0465E-04&  1.9429  &6.7758E-04&2.0146 \\
            $32\times 32\times 32$ &  3.9635E-04&  1.9271& 1.5851E-04& 1.9315   & 1.6895E-04& 2.0038 \\
            $64\times 64\times 64$&   1.0182E-04&  1.9608&   4.0875E-05&  1.9553  &4.2206E-05& 2.0011
             \\ [1ex]
            \hline
        \end{tabular}
     \end{table}     }
\end{example}
\begin{example}\label{2}(Solution positivity)
We consider the two-species PNP system with initial data of form 
 \begin{equation}\label{Pos}
 \left \{
\begin{array}{rl}
 \hfill  \partial_t \rho_1 =&\nabla \cdot(\nabla \rho_1 +  \rho_1 \nabla \phi),  \hfill \ \ \  \mathbf{x} \in \Omega, \ t>0,\\
 \hfill  \partial_t \rho_2 =&\nabla \cdot(\nabla \rho_2 -  \rho_2 \nabla \phi),  \hfill \ \ \   \mathbf{x} \in \Omega, \ t>0,\\
 \hfill  -\Delta \psi =&\rho_1 -   \rho_2 +10\chi_{_{[0.2,0.4] \times [0.2, 0.4] \times [0.2, 0.4]}},  \hfill \ \ \   \mathbf{x} \in \Omega, \ t>0,\\
 \hfill  \rho_1^{in}(\mathbf{x})= & \chi_{_{[0,0.25] \times [0, 0.25] \times [0, 0.25]}},\\
 \hfill  \rho_2^{in}(\mathbf{x})= & 2\chi_{_{[0,0.25] \times [0, 0.25] \times [0, 0.25]}}.
\end{array}
\right.
\end{equation}
This corresponds to (\ref{PNP0}) with $D_1=D_2=1,$ $q_1=-q_2=1,$ $k_BT=1,$ $\epsilon(\mathbf{x})=4\pi$, $\mu_i=0$, and $f(\mathbf{x})=10\chi_{_{[0.2,0.4] \times [0.2, 0.4] \times [0.2, 0.4]}}.$ 

With $
 \partial \Omega_D=\{\mathbf{x}\in\bar  \Omega : y=0,1\},
$
and $\partial \Omega_N=\partial \bar \Omega \setminus \partial \Omega_D$,  we solve the problem subject to mixed boundary conditions 
  \begin{equation}\label{Mixbc}
 \left \{
\begin{array}{rl}
 \hfill  (\nabla \phi)\cdot \mathbf{n}  = & 0, \  (\nabla \rho_1  +  \rho_1 \nabla \phi)\cdot \mathbf{n}=0 , \   (\nabla \rho_2 -  \rho_2 \nabla \phi)\cdot \mathbf{n}=0 ,  \hfill \ \ \   \mathbf{x} \in \partial \Omega_N, \\
 \hfill   \phi^b (\mathbf{x},t) = & (x^2(1-x)^2+z^2(1-z)^2)e^{-t},  \hfill \ \ \   \mathbf{x} \in \partial \Omega_D, \\
\hfill   \rho^b_1 (\mathbf{x},t) = &4x^2(1-x)^2e^{-t} ,   \hfill \ \ \   \mathbf{x} \in \partial \Omega_D,  \\
\hfill     \rho^b_2  (\mathbf{x},t)=& z^2(1-z)^2e^{-t},  \hfill \ \ \   \mathbf{x} \in \partial \Omega_D. 
\end{array}
\right.
\end{equation}

We use $30\times 30\times 30$ cells with $\tau=0.5h$ to compute numerical solutions up to $t=2$. Given in Fig.1 
are the time evolution of numerical solutions (top three rows) and the minimum of $\rho_1, \rho_2$ (bottom row) obtained by the scheme (\ref{fully1C}), showing non-negative approximations for both $\rho_1$ and $\rho_2$. Results obtained by the scheme (\ref{fully2Ci}) are given in Fig.2. Note that the positivity preserving limiter keeps being invoked when we use the scheme (\ref{fully2Ci}). The CPU time (average of 10 simulations) needed for running the schemes (\ref{fully1C}) and (\ref{fully2Ci}) are $207.27$ seconds and $203.15 $ seconds, respectively, from which we see that the second-order scheme is as efficient as the first order scheme.

\begin{figure}[h]
	\caption{Example \ref{2}:  $\rho_1, \rho_2, \phi$ computed by scheme (\ref{fully1C})}
	\centering  
	\subfigure{\includegraphics[width=0.25\linewidth]{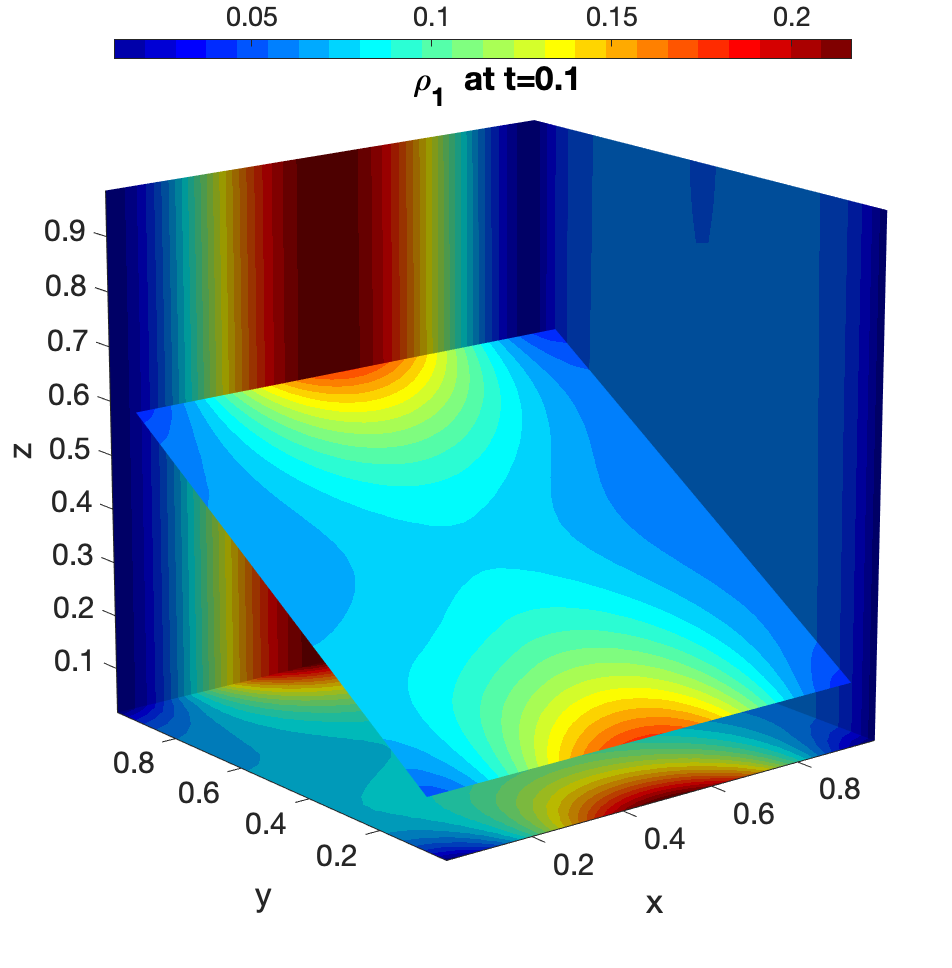}}
	\subfigure{\includegraphics[width=0.25\linewidth]{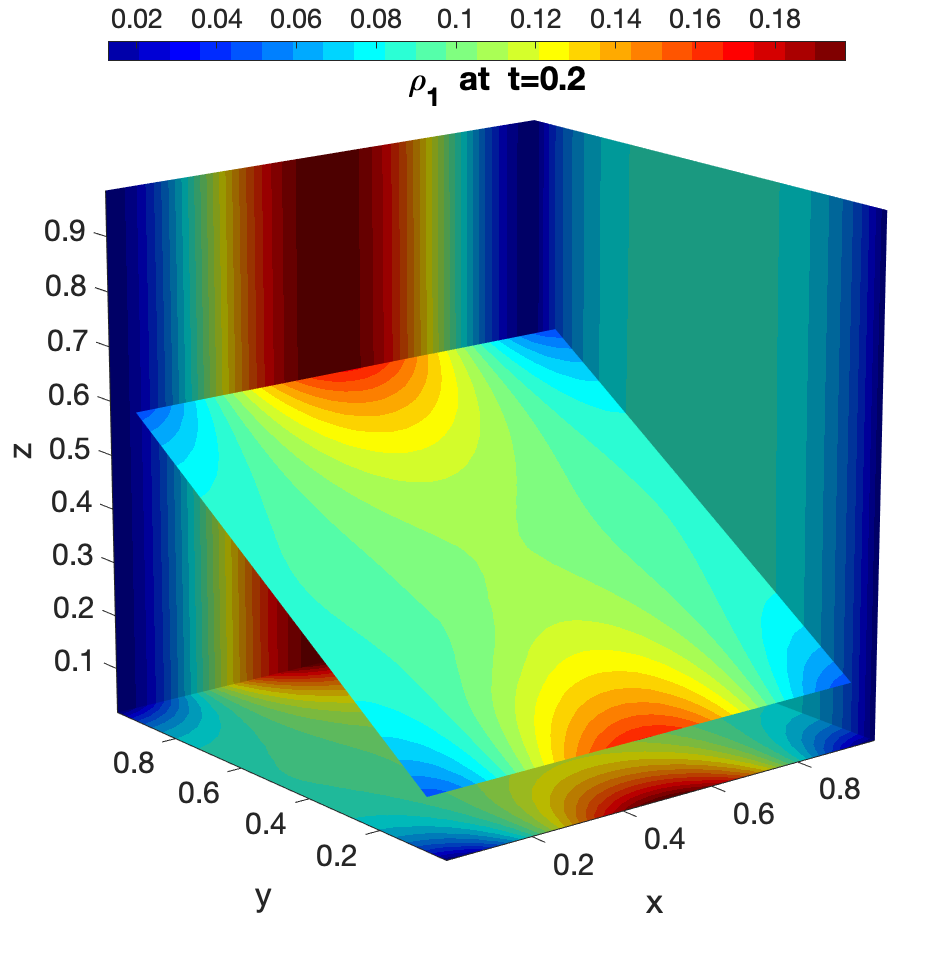}}
	\subfigure{\includegraphics[width=0.25\linewidth]{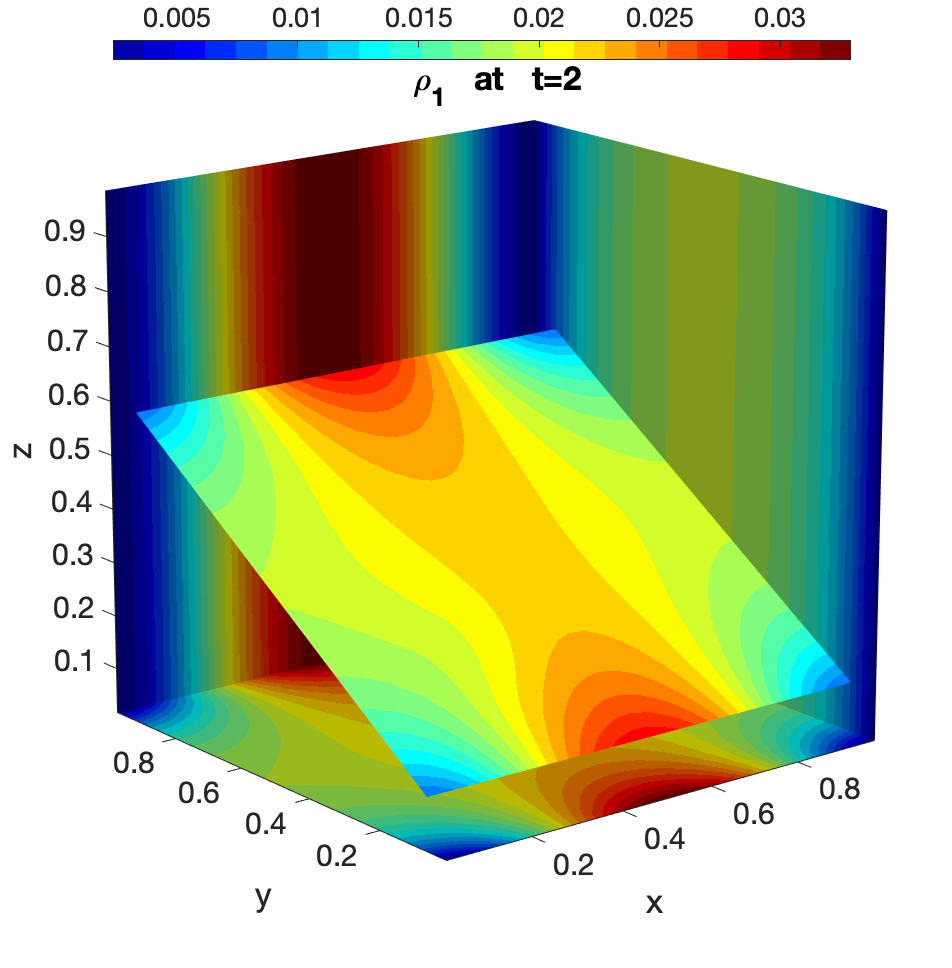}}
	\subfigure{\includegraphics[width=0.25\linewidth]{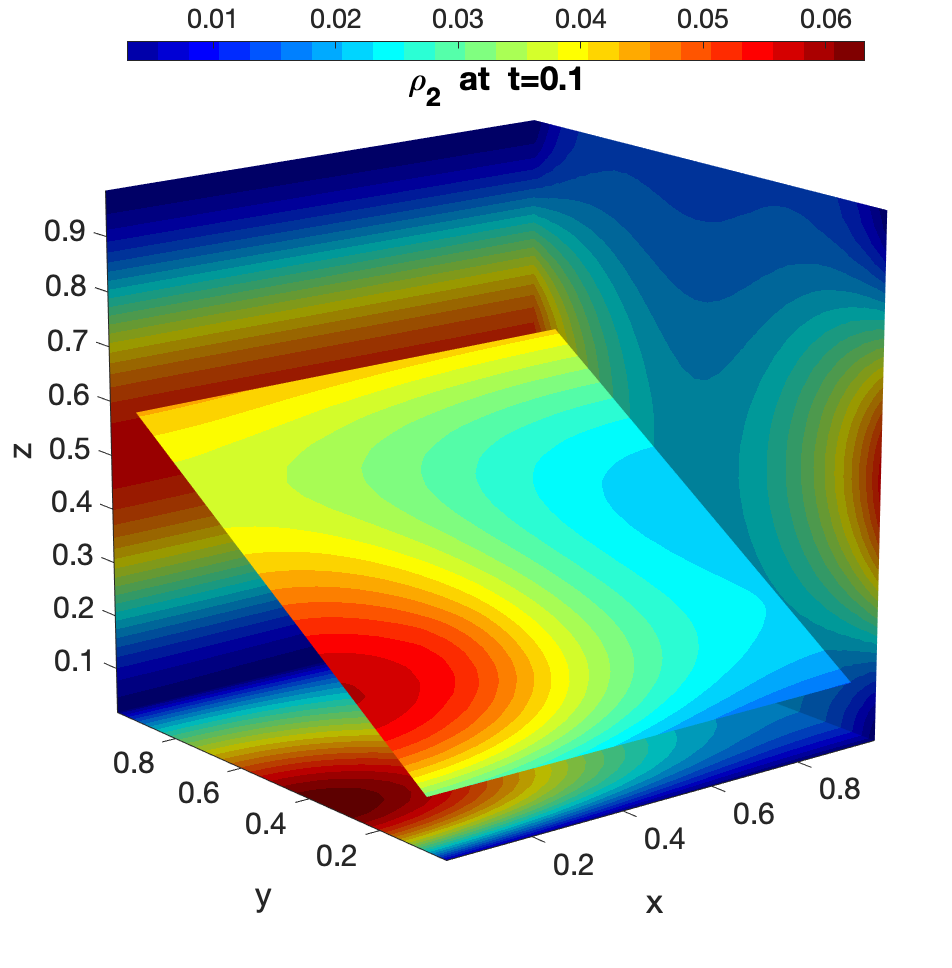}}
	\subfigure{\includegraphics[width=0.25\linewidth]{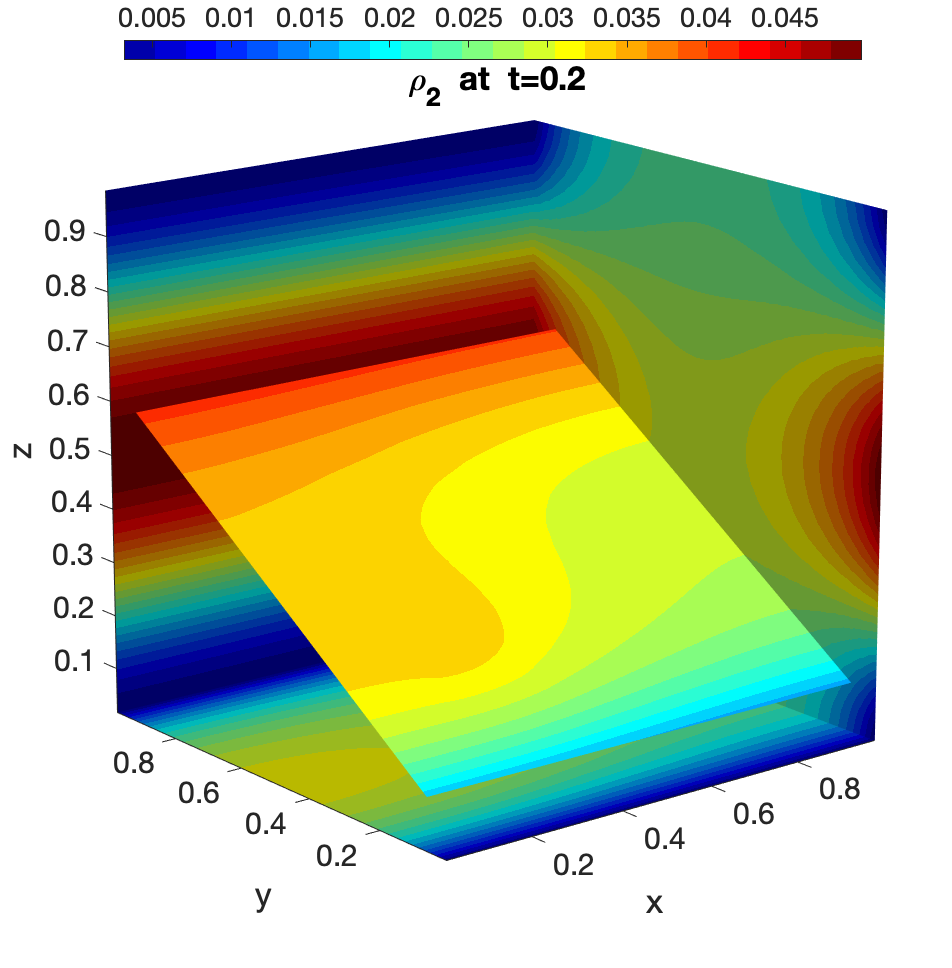}}
	\subfigure{\includegraphics[width=0.25\linewidth]{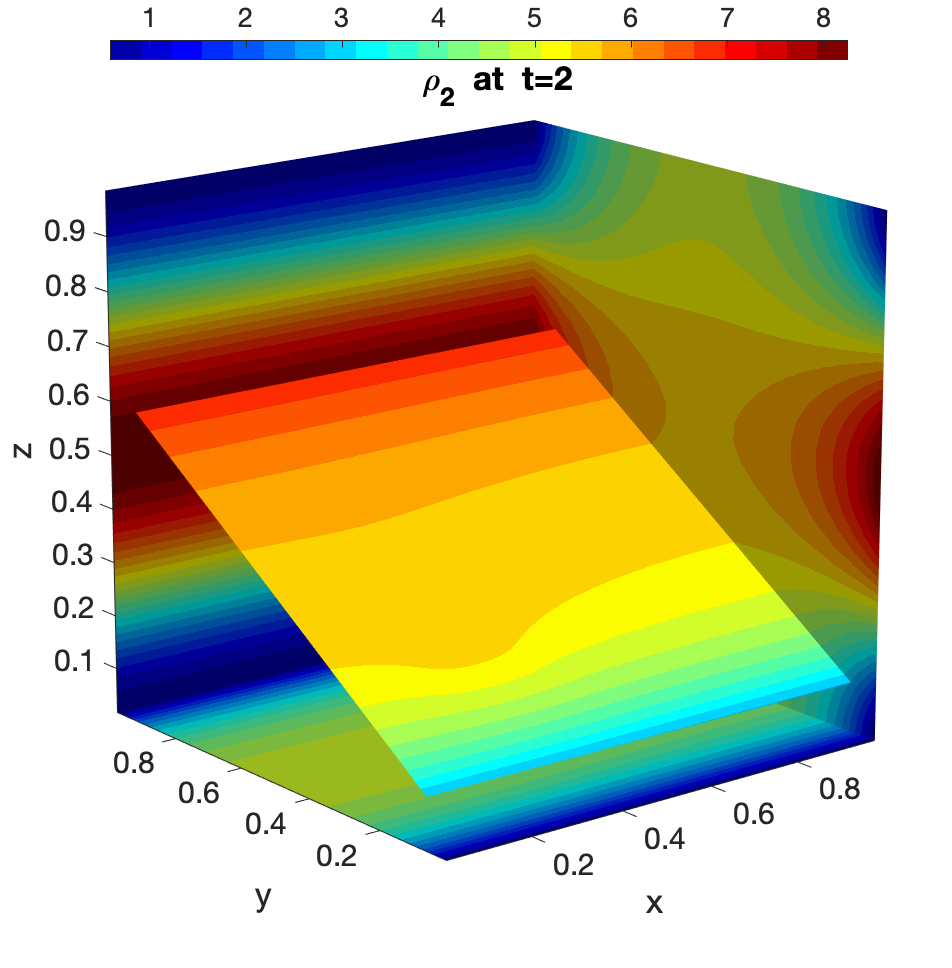}}
	\subfigure{\includegraphics[width=0.25\linewidth]{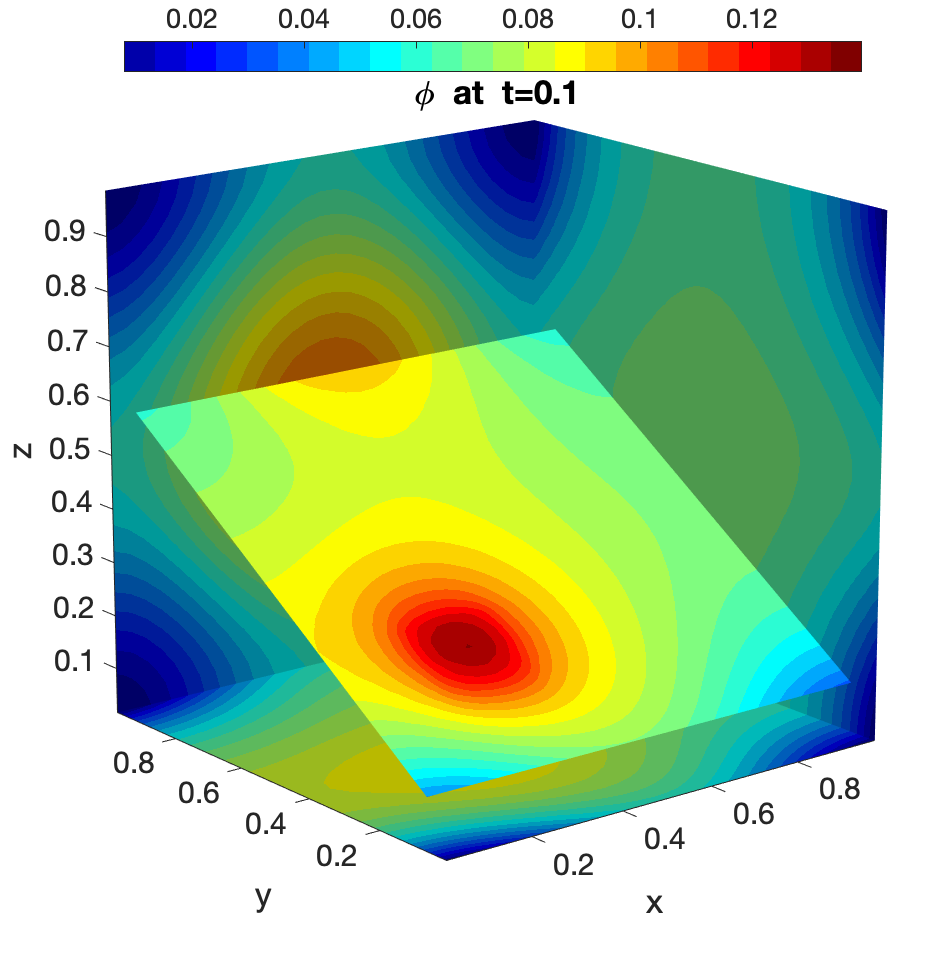}}
	\subfigure{\includegraphics[width=0.25\linewidth]{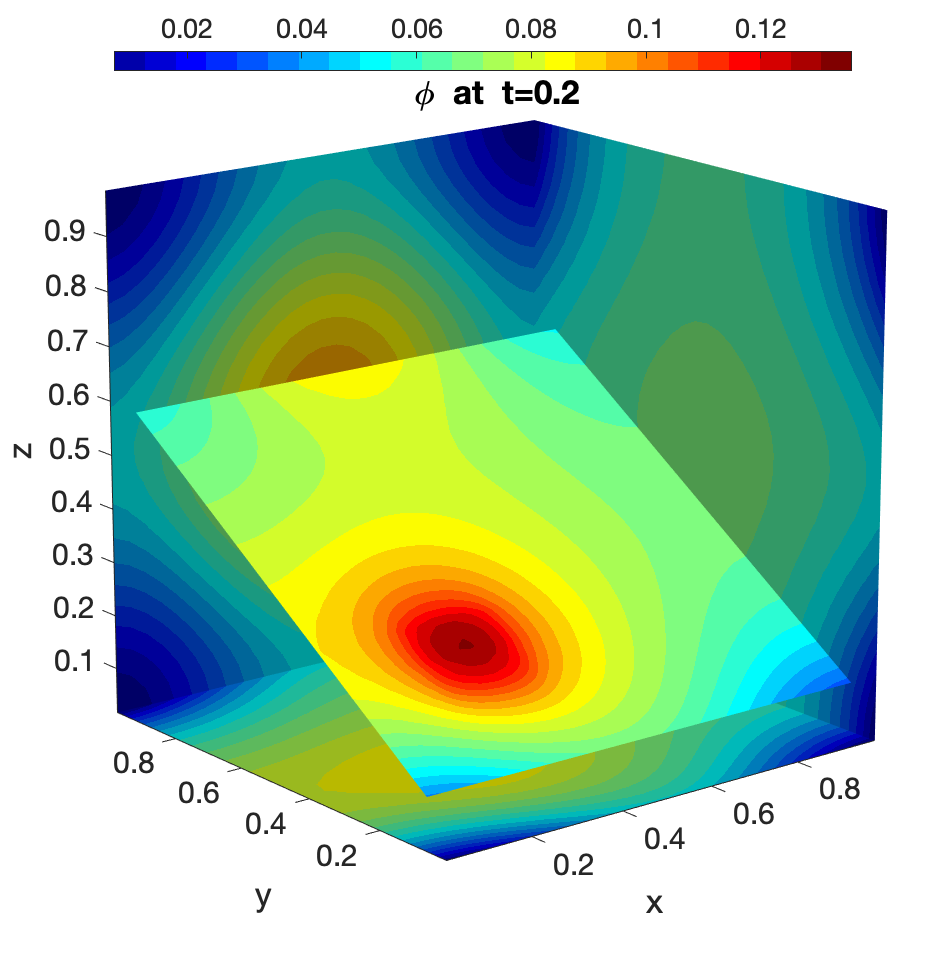}}
	\subfigure{\includegraphics[width=0.25\linewidth]{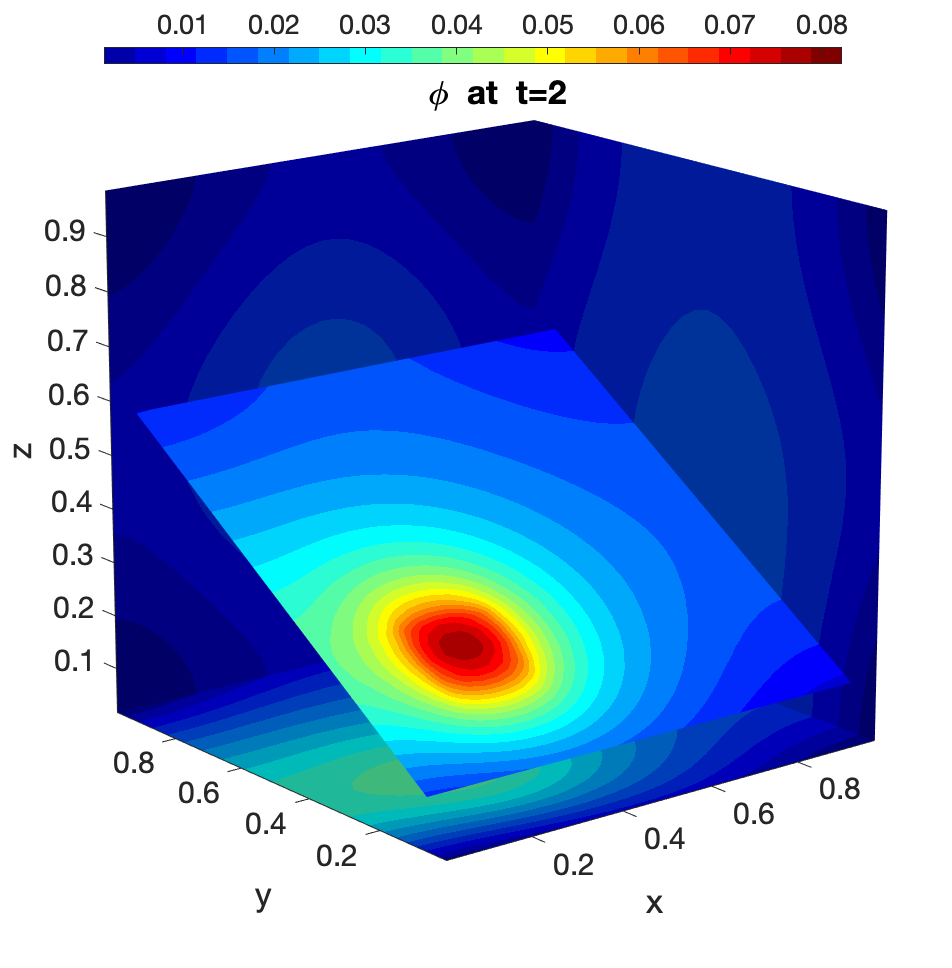}}
	\subfigure{\includegraphics[width=0.3\linewidth]{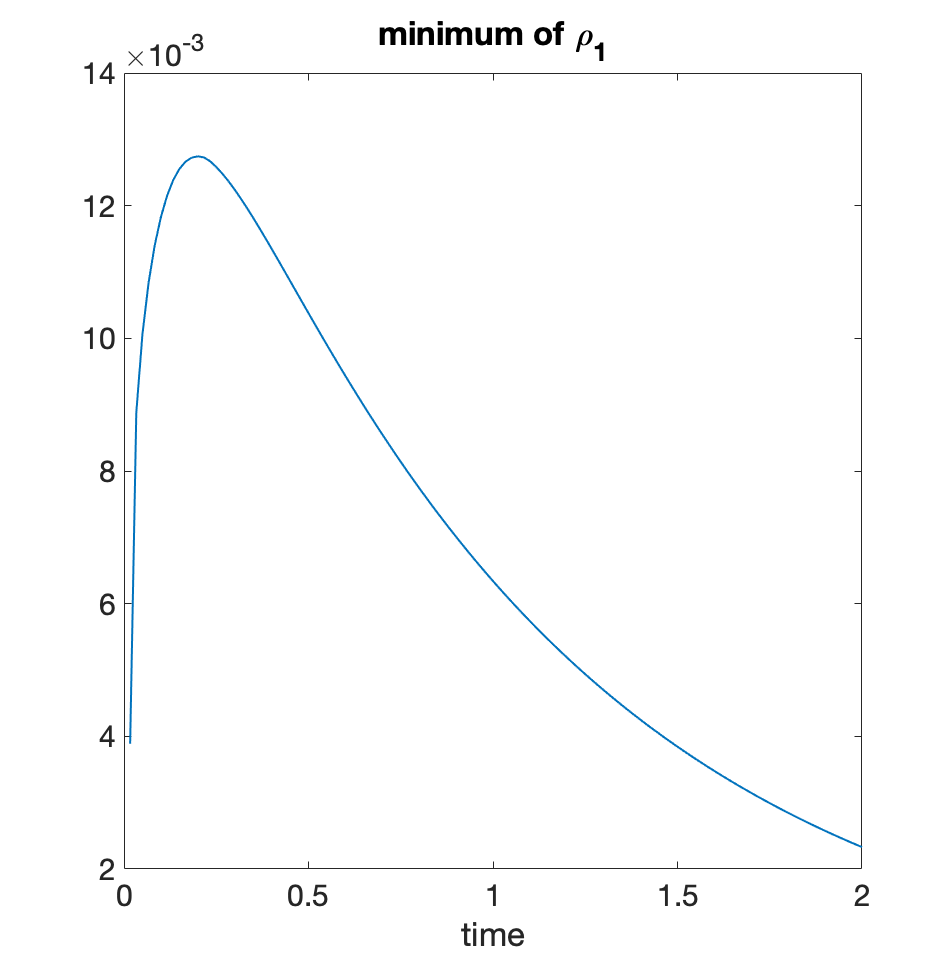}}
	\subfigure{\includegraphics[width=0.3\linewidth]{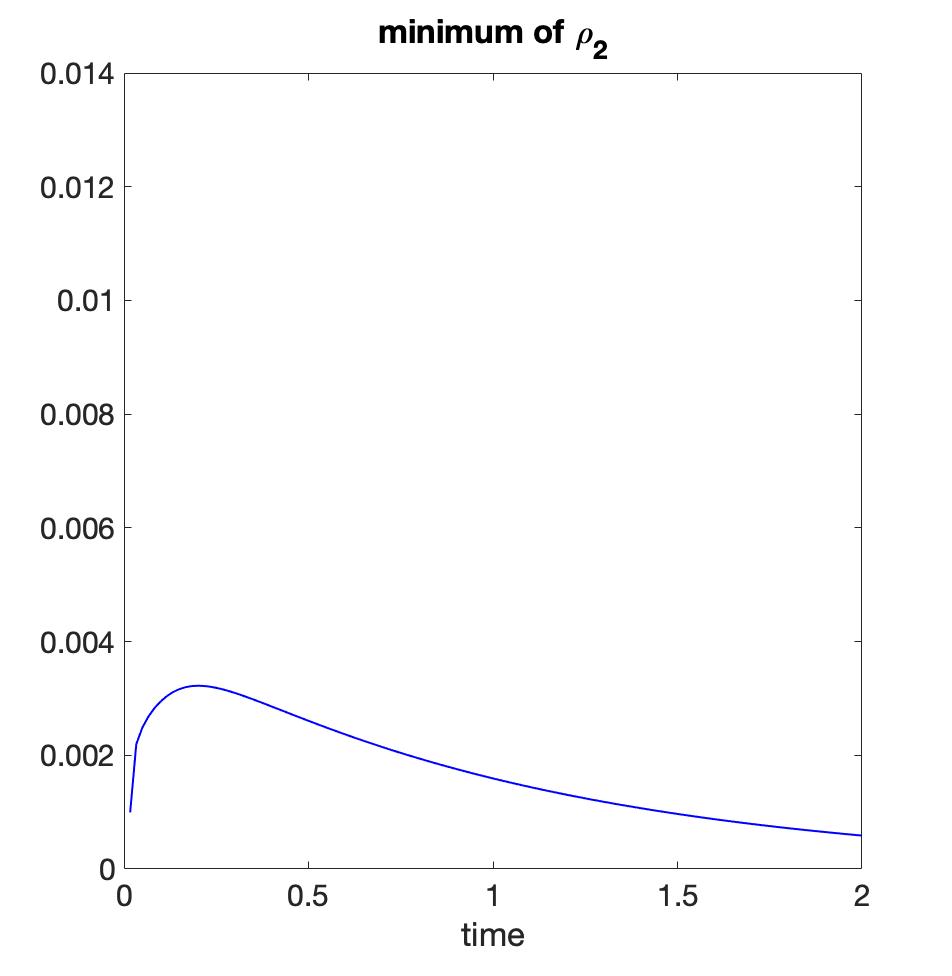}}
\end{figure}

\begin{figure}[h]
	\caption{Example \ref{2}:  $\rho_1, \rho_2, \phi$ computed by scheme (\ref{fully2Ci}) ( with limiter)}
	\centering  
	\subfigure{\includegraphics[width=0.25\linewidth]{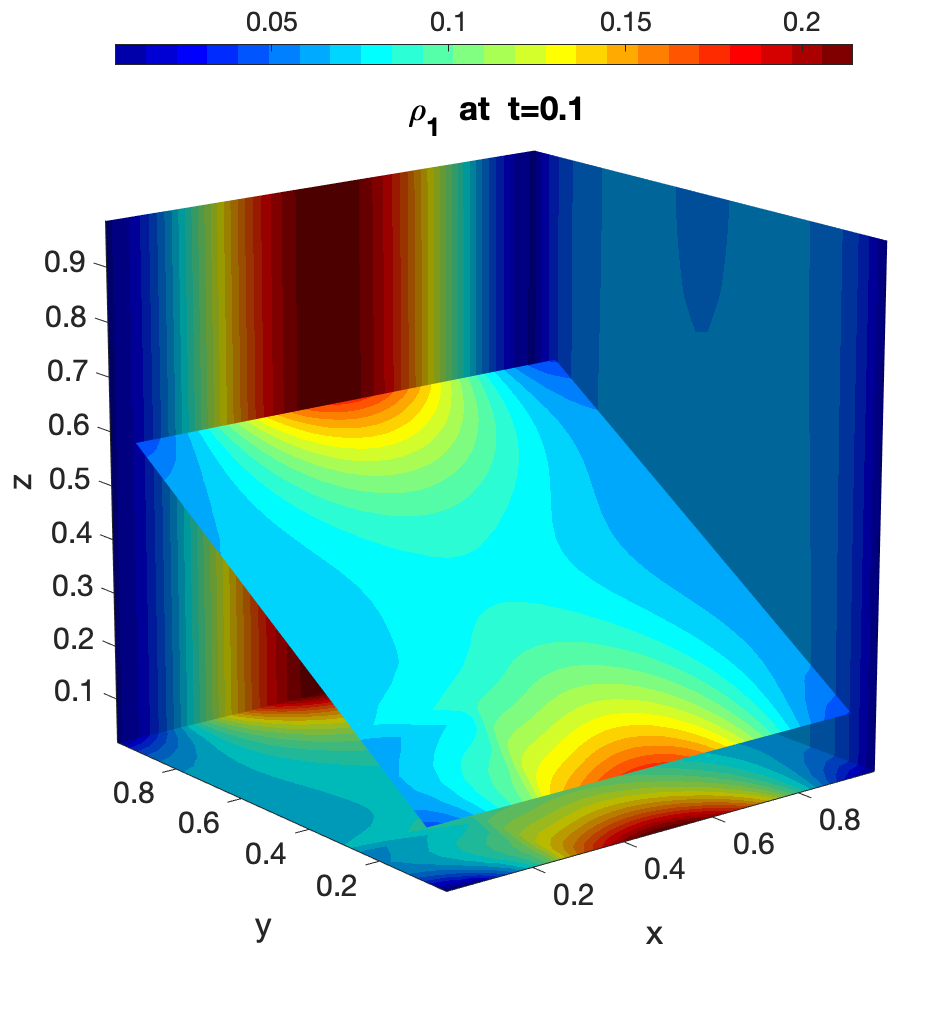}}
	\subfigure{\includegraphics[width=0.25\linewidth]{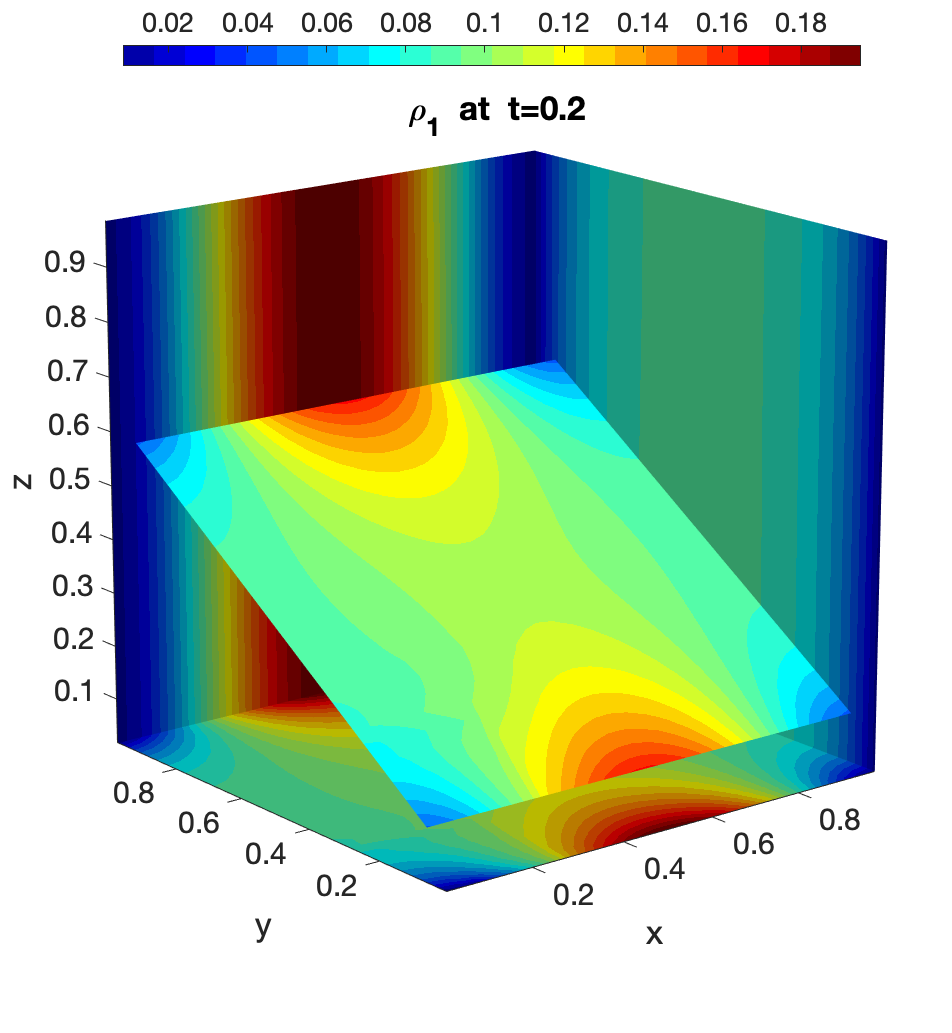}}
	\subfigure{\includegraphics[width=0.25\linewidth]{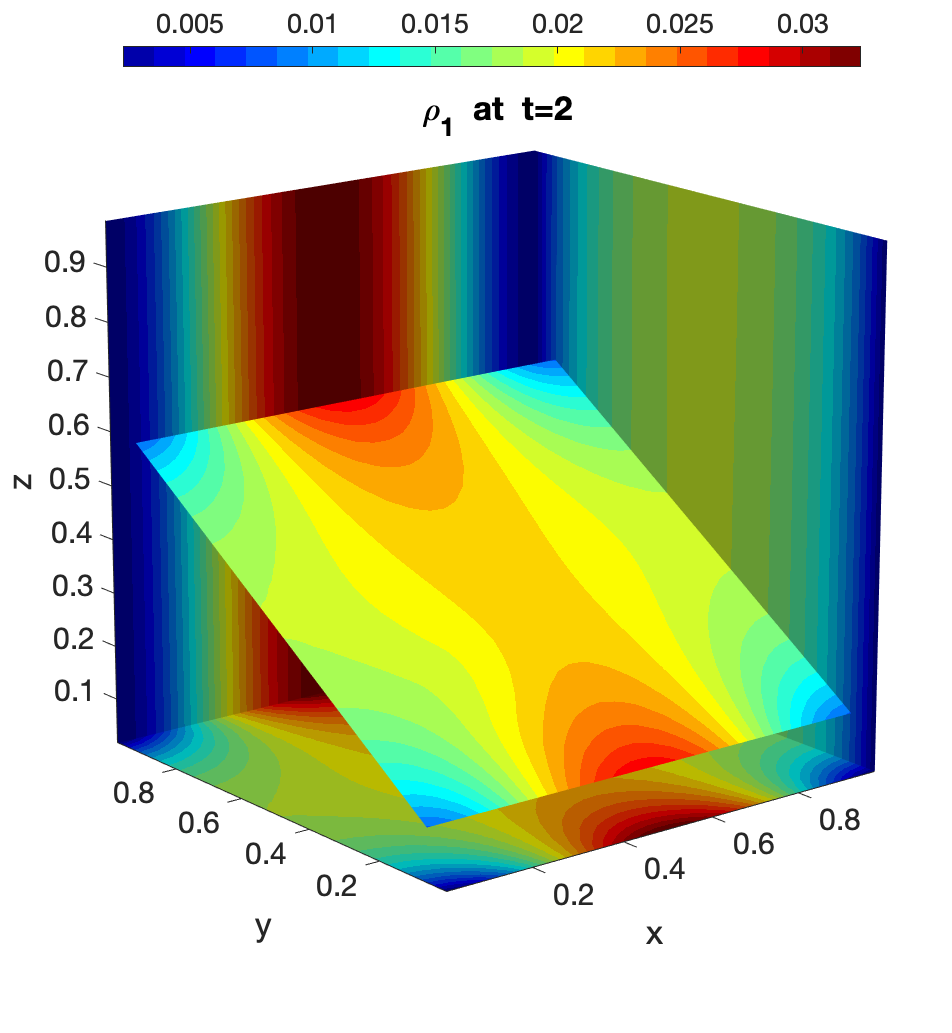}}
	\subfigure{\includegraphics[width=0.25\linewidth]{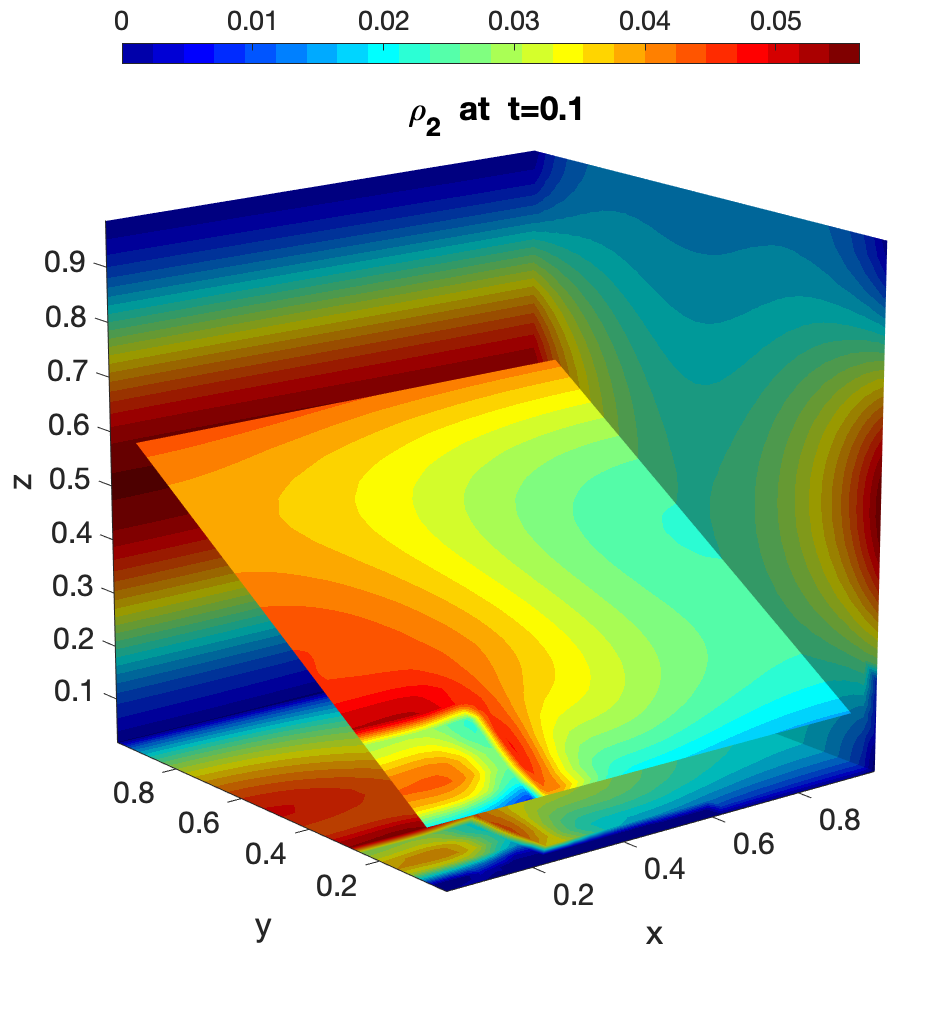}}
	\subfigure{\includegraphics[width=0.25\linewidth]{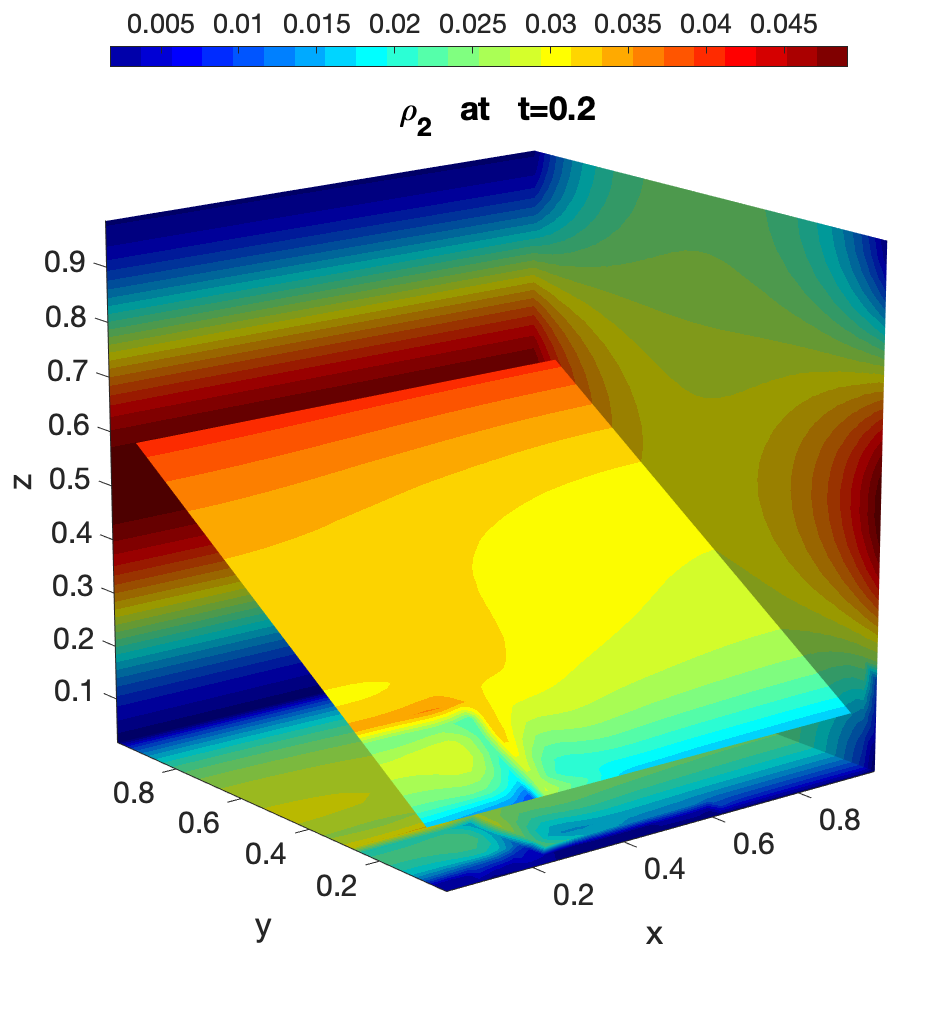}}
	\subfigure{\includegraphics[width=0.25\linewidth]{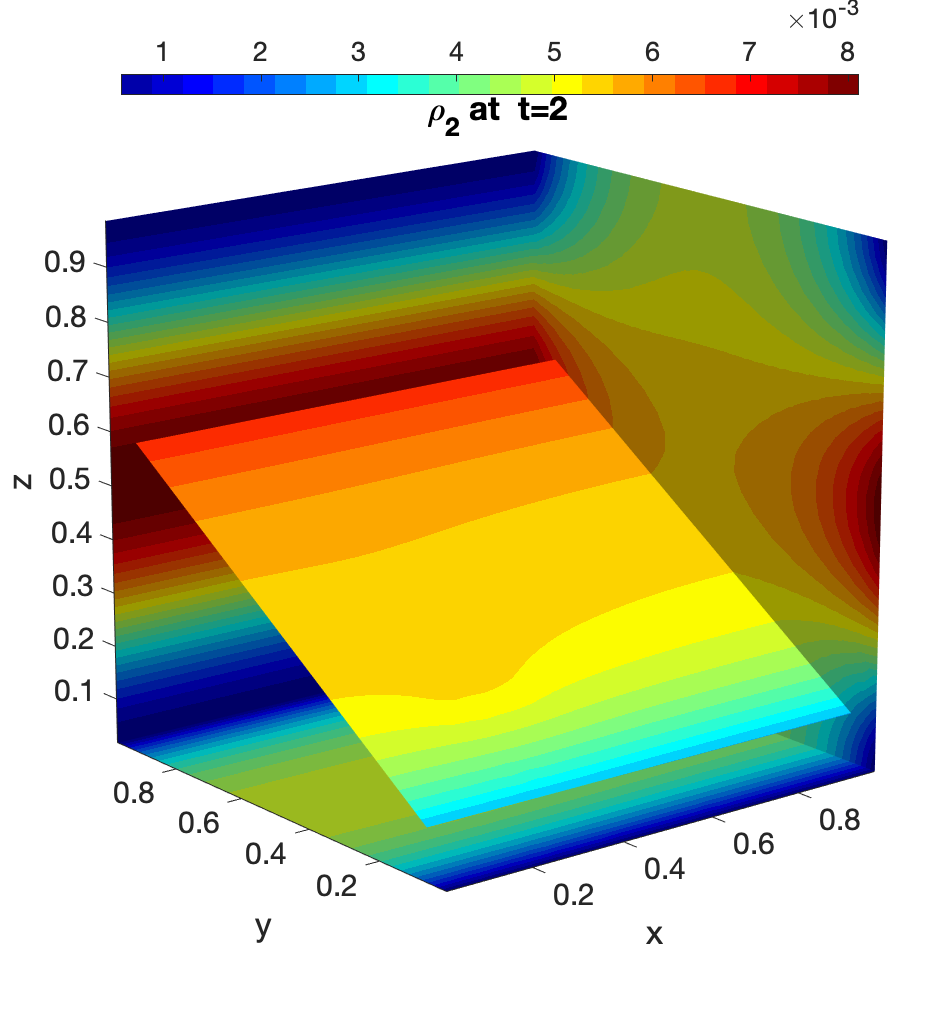}}
	\subfigure{\includegraphics[width=0.25\linewidth]{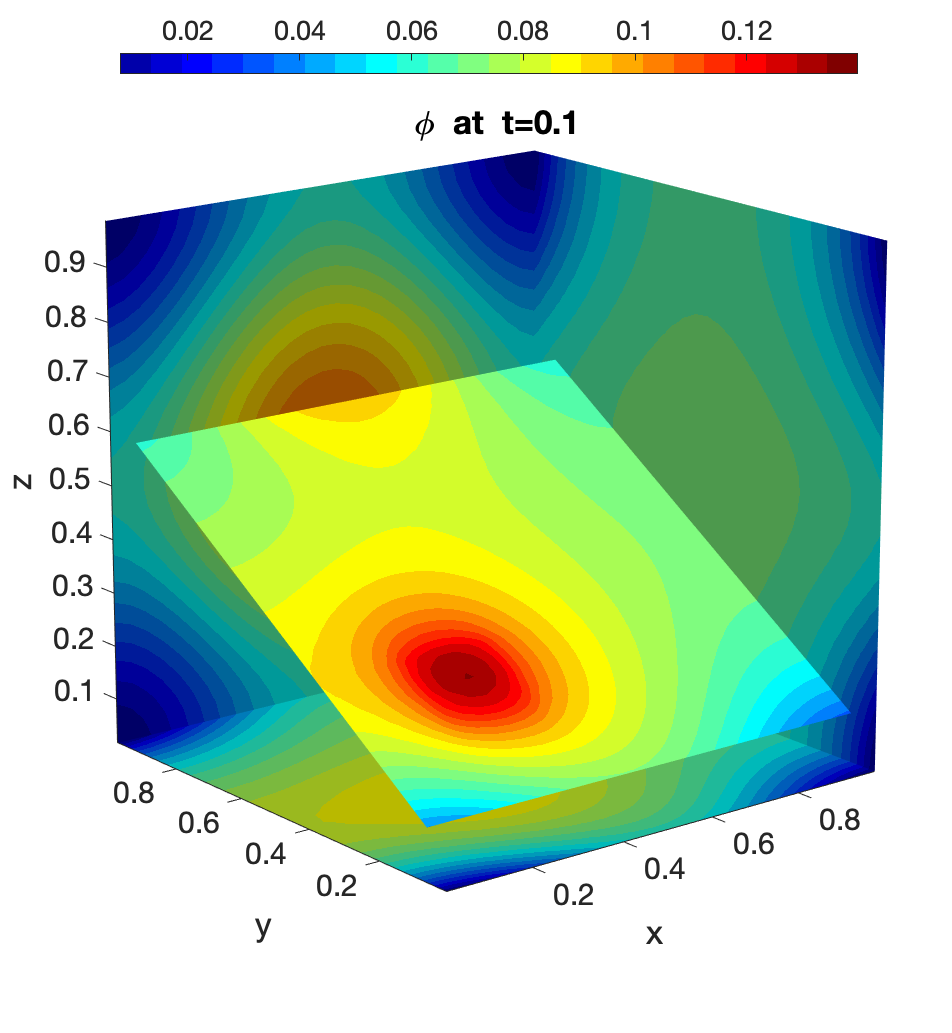}}
	\subfigure{\includegraphics[width=0.25\linewidth]{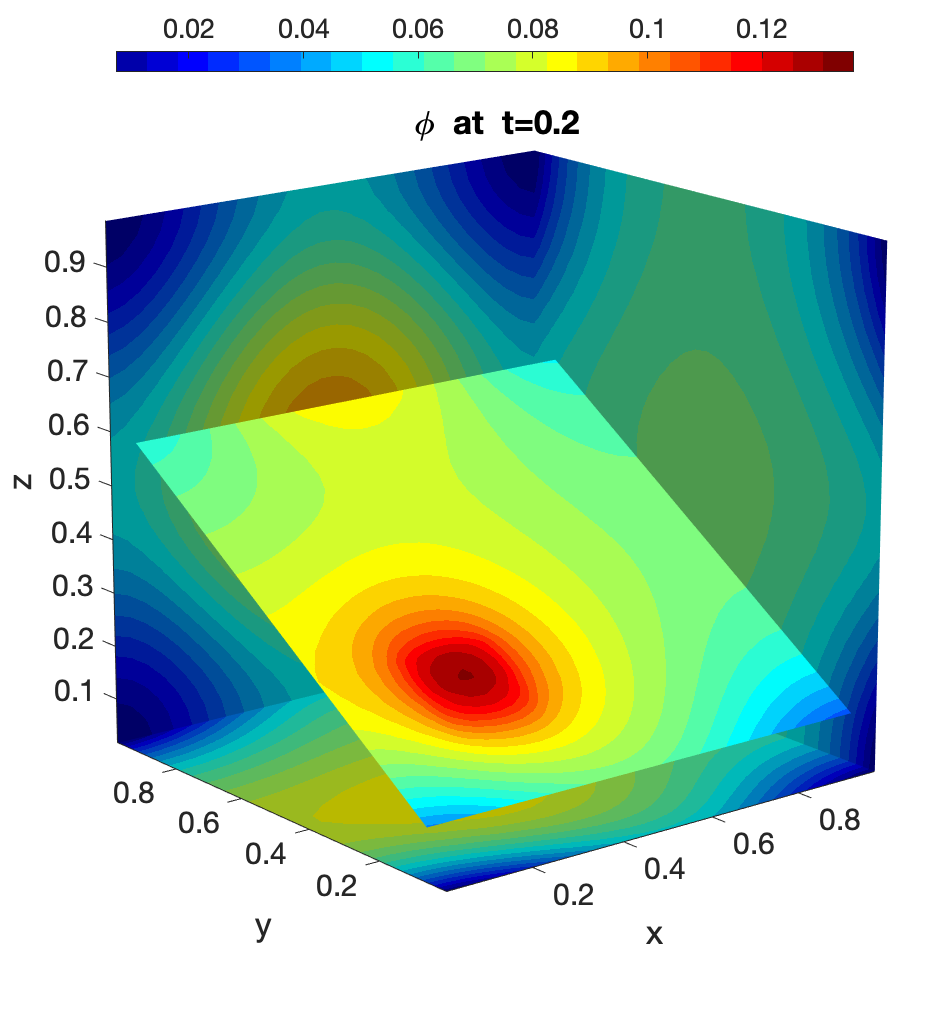}}
	\subfigure{\includegraphics[width=0.25\linewidth]{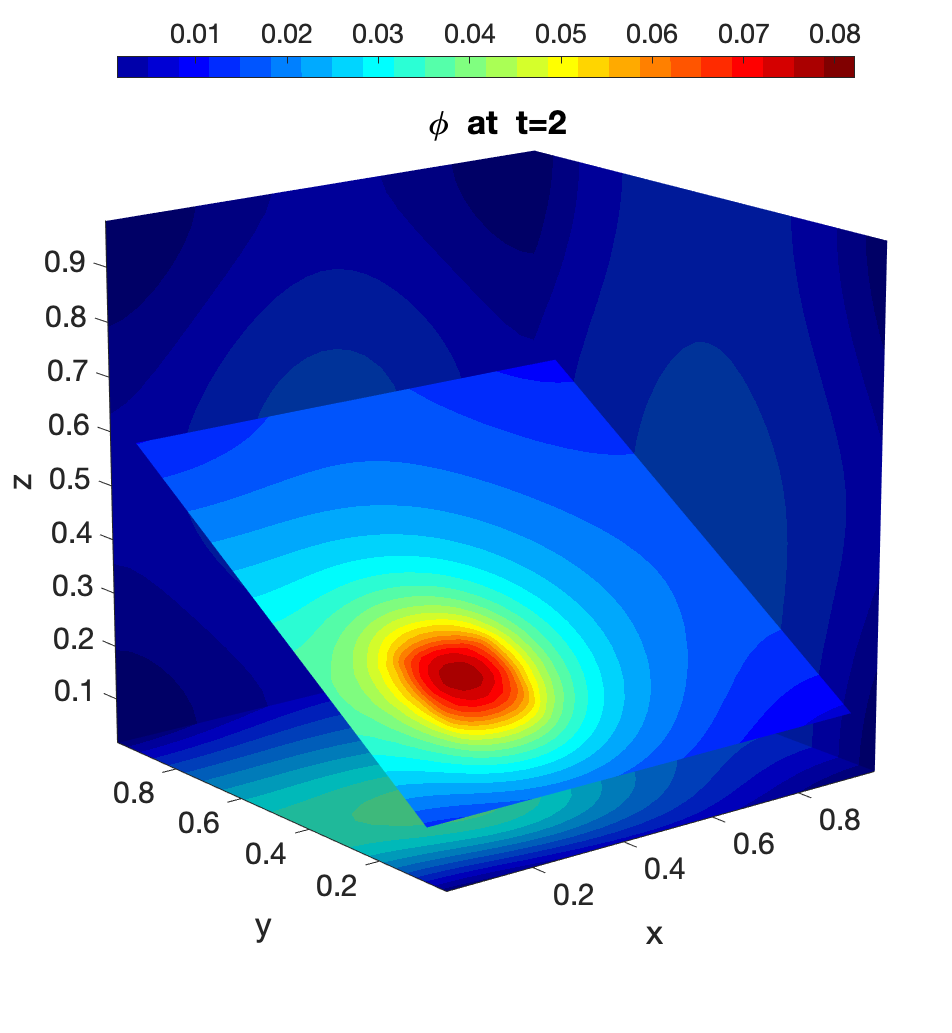}}
	\subfigure{\includegraphics[width=0.3\linewidth]{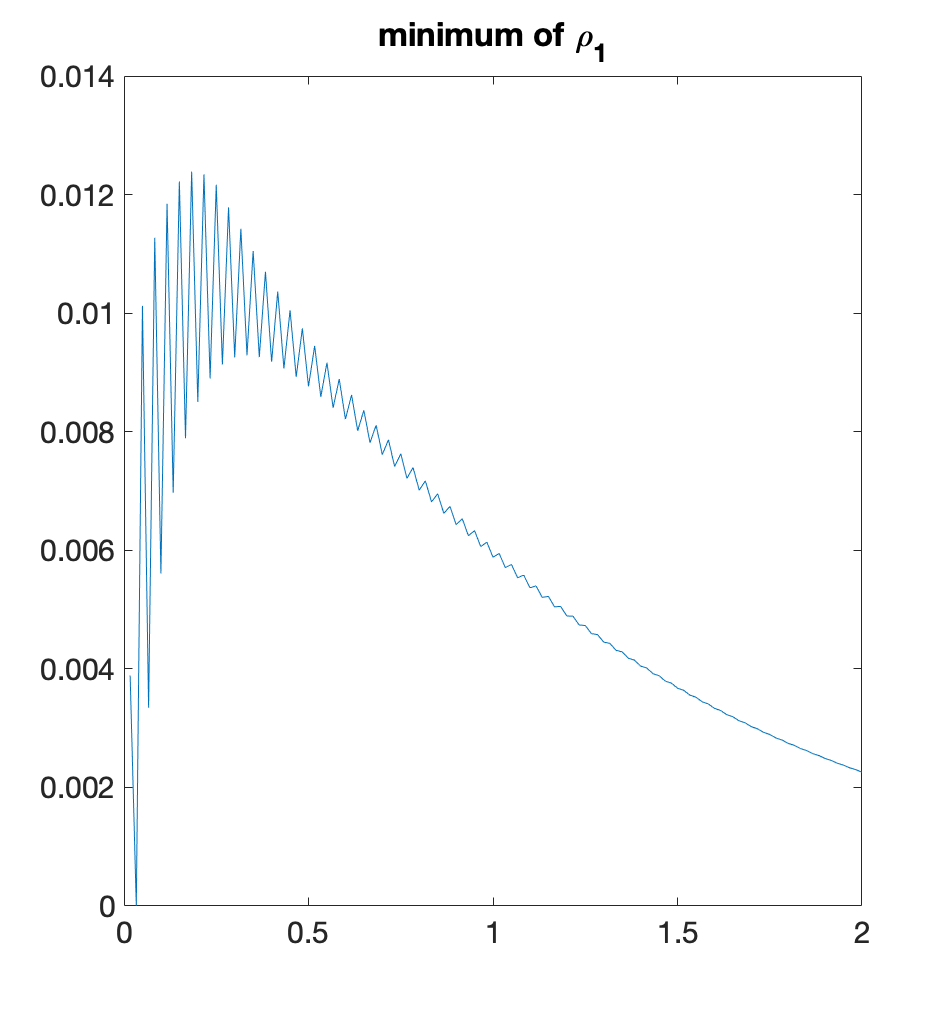}}
	\subfigure{\includegraphics[width=0.3\linewidth]{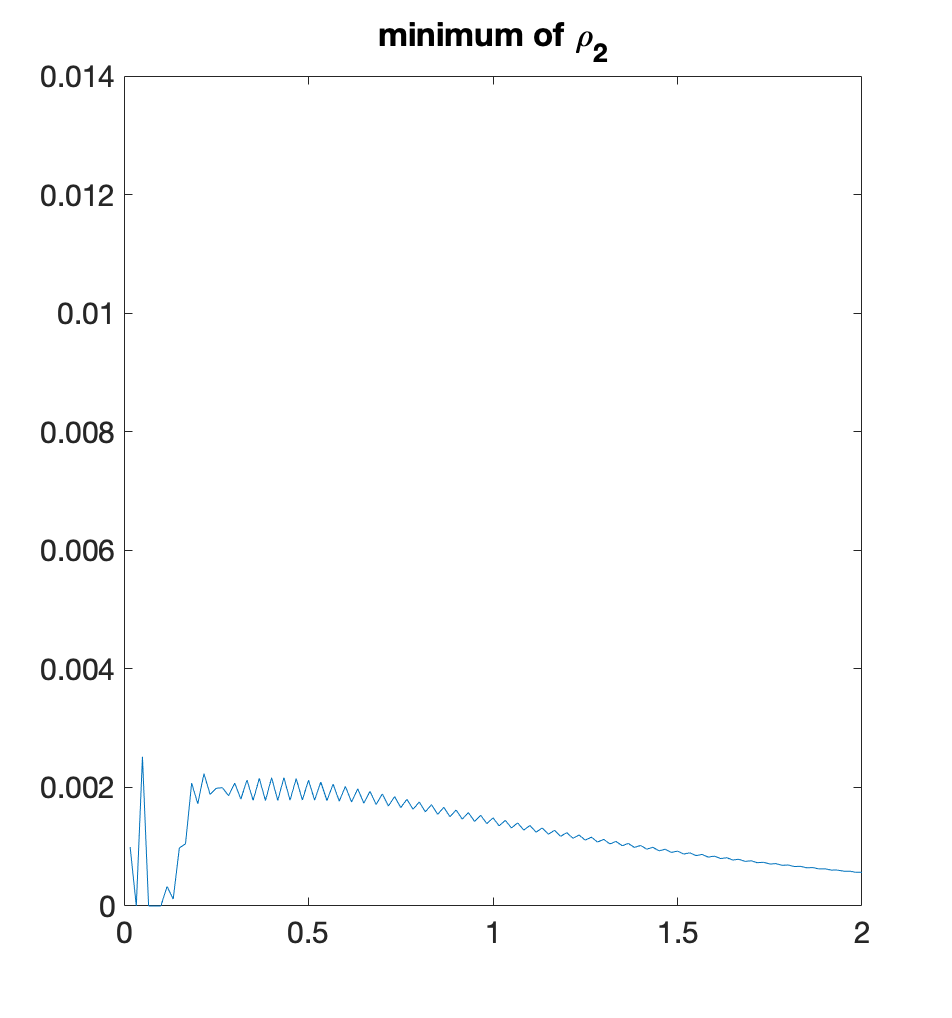}}
\end{figure}
\end{example}

\begin{example}\label{3}(Mass conservation and energy dissipation)
In this numerical example we test both mass conservation and energy dissipation properties of our schemes. 

We consider system (\ref{Pos}) with zero flux boundary conditions:
$$
(\nabla \phi)\cdot \mathbf{n}  =  0, \  (\nabla \rho_1  +  \rho_1 \nabla \phi)\cdot \mathbf{n}=0 , \   (\nabla \rho_2 -  \rho_2 \nabla \phi)\cdot \mathbf{n}=0 ,  \hfill \ \ \   x \in \partial \Omega.
$$
Numerical approximations  of $\rho_1$ and $\rho_2$ at $t=2$ obtained by the scheme  (\ref{fully1C}) are given in Fig.3. We can see by comparing Fig.3  and  Fig.1 that boundary conditions have strong effects 
on the solution profiles. In Fig.4 (left) are the time evolution of the total mass and free energy obtained by the scheme (\ref{fully1C}), the results verify our theoretical findings in Theorem \ref{First-Energy}. In Fig.4 (right) are plots of the free energy and total mass obtained by (\ref{fully2Ci}). In this test the second order scheme looks also energy dissipative and mass conservative. 
\begin{figure}[h]
	\caption{Example \ref{3}:  $\rho_1, \rho_2$ computed by scheme (\ref{fully1C}) }
	\centering  
	\subfigure{\includegraphics[width=0.35\linewidth]{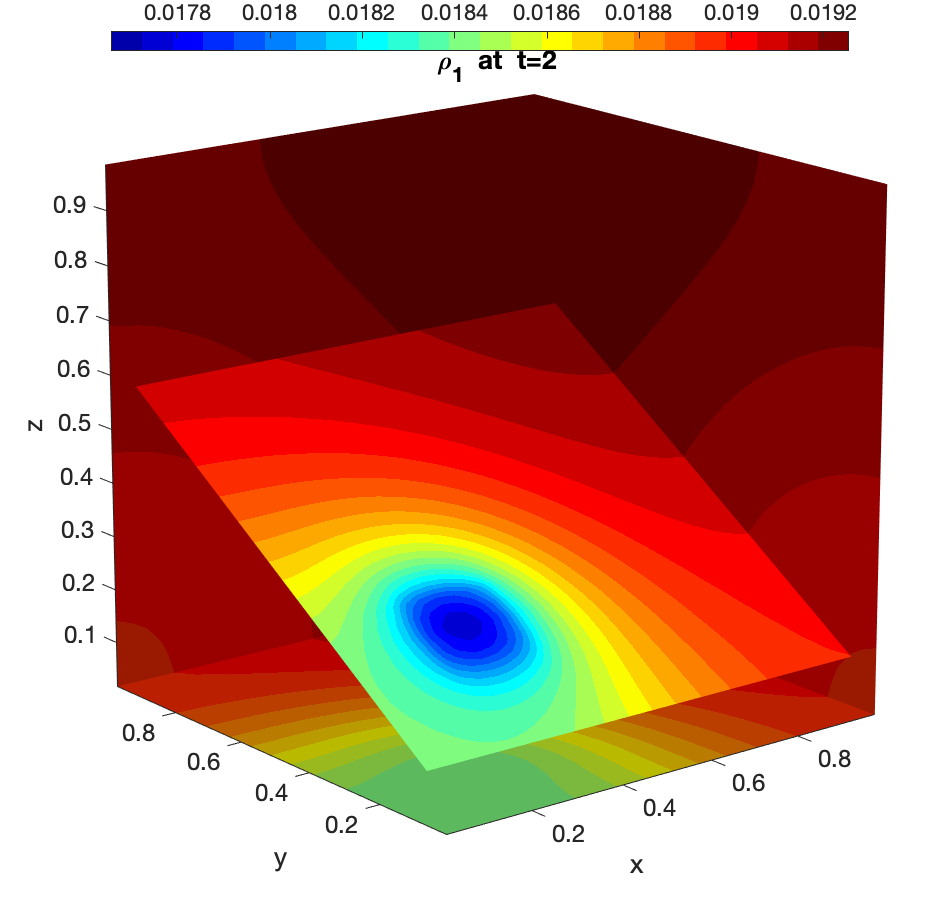}}
	\subfigure{\includegraphics[width=0.35\linewidth]{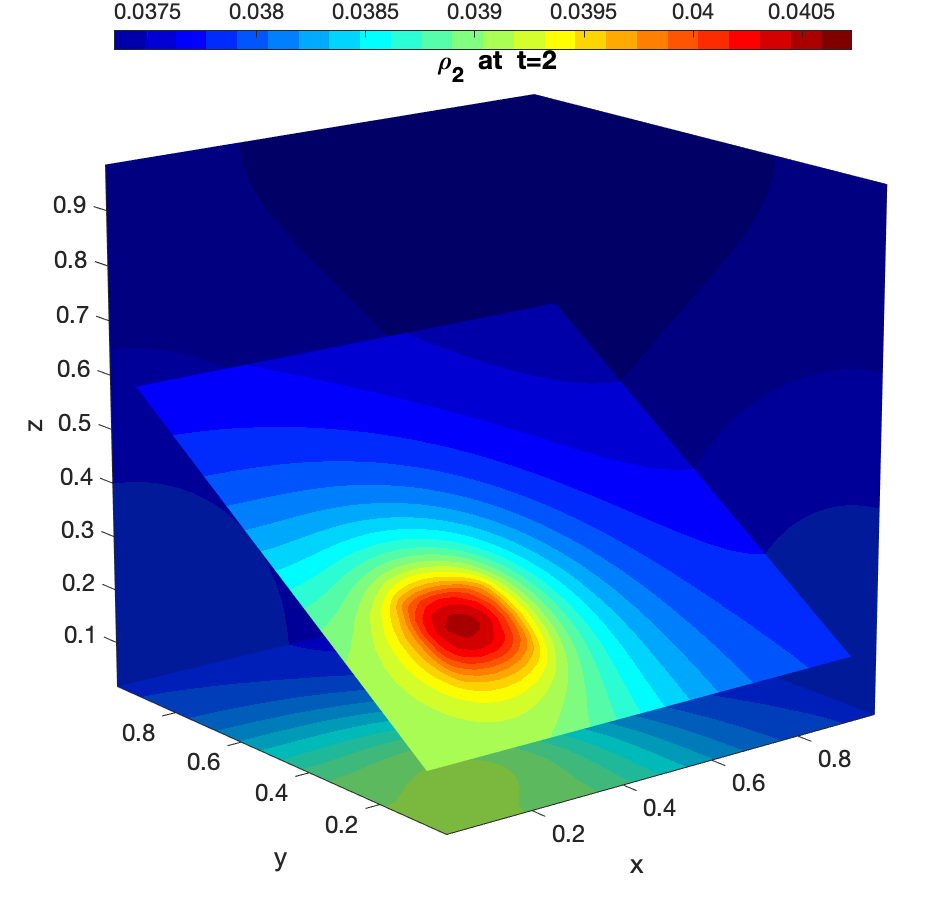}}
	\caption{Example \ref{3}:  Mass conservation and energy dissipation}
	\centering  
	\subfigure{\includegraphics[width=0.35\linewidth]{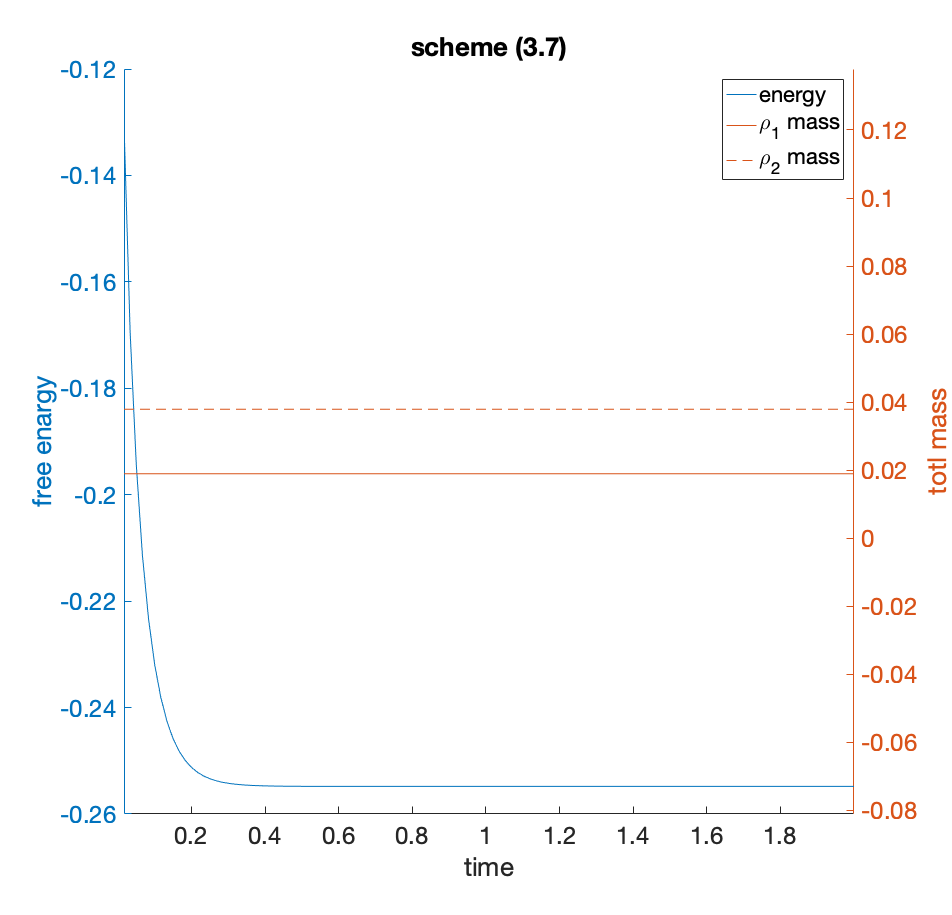}}
	\subfigure{\includegraphics[width=0.35\linewidth]{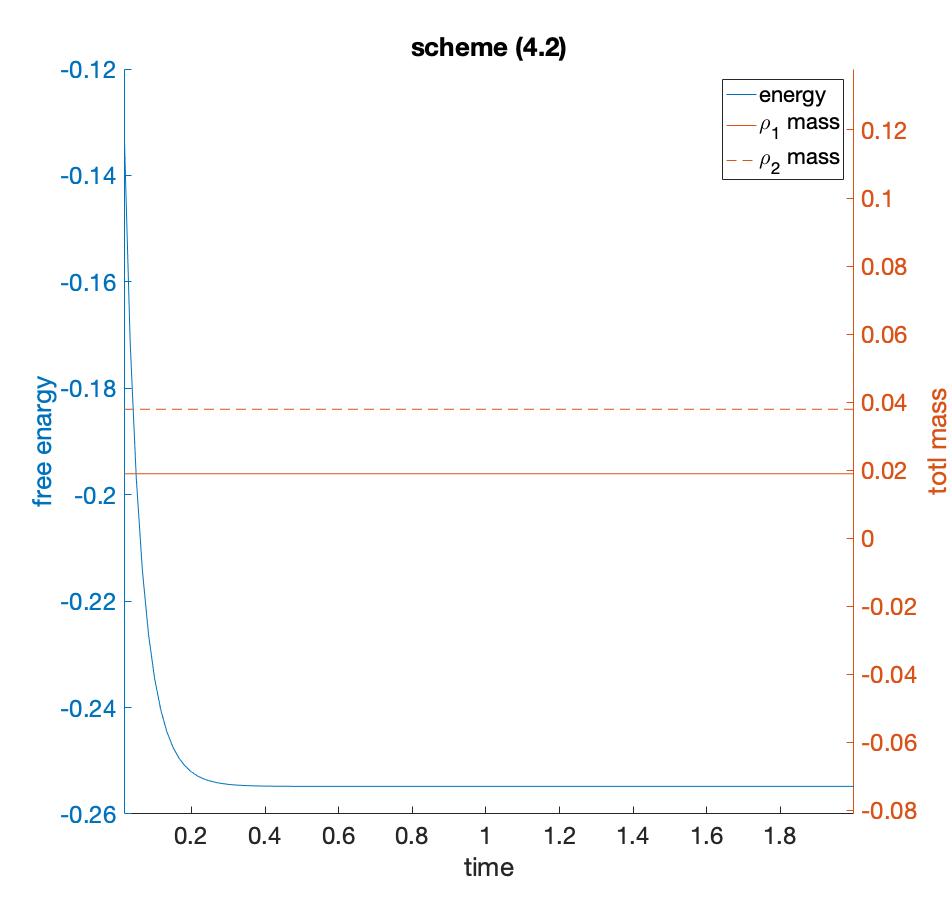}}
\end{figure}
\end{example}

\section{Concluding remarks}\label{sec6}
In this paper, we have developed unconditional structure-preserving schemes for PNP equations in more general settings. These schemes are shown to preserve several important physical laws at the fully discrete level including: mass conservation, solution positivity, and free energy dissipation. The non-logarithmic Landau reformulation of the model is important, enabling us to construct a simple, easy-to-implement fully discrete scheme (first order in time, second order in space), which proved to satisfy all three desired properties of the continuous model with
only $O(1)$ time step restriction. We further designed a second order (in both time and space) scheme, which has the same computational complexity as the first order scheme. For such second order scheme, we employed a local scaling limiter to restore solution positivity where necessary. Moreover, we rigorously proved  that the limiter does not destroy the desired accuracy.  Three-dimensional numerical tests are conducted to evaluate the scheme performance and verify our theoretical findings. 
Our schemes presented with $\mu_i$ given can be applied to complex chemical potentials without difficulty.  

\section*{Acknowledgments}  This research was partially supported by the National Science Foundation under Grant DMS1812666.

\appendix
\section{Proof of Theorem \ref{Semi-WELL}}
\begin{proof}
The elliptic problem (\ref{elliptic1}) can be rewritten in $w= \rho^{n+1}_i e^{\psi^n_i}$ as 
\begin{subequations}\label{EE1}
\begin{align}
\label{EL1}
& e^{-\psi^n_i}w-\tau\nabla \cdot (D_i(x)e^{-\psi^n_i}\nabla w)=\rho^n_i, \\
\label{EL2}
& w=\rho^b_i(x,t_{n+1})e^{\psi^b_i(x,t_n)} , \quad x\in \partial \Omega_D,\\
\label{EL3}
&  (\nabla w)\cdot \mathbf{n}=0, \quad x\in \partial \Omega_N.
\end{align}
\end{subequations}
Let  $\gamma_0$ be the trace operator on $\partial \Omega_D$. The above problem admits a variational formulation of form 
\begin{align}\label{BL}
B[u, v]=Lv, \quad u, v \in H,
\end{align}
where for a Dirichlet lift $G \in H^2(\Omega)$ with trace  $\gamma_0(G)=\rho^b_i(x,t_{n+1})e^{\psi^b_i(x,t_n)}$, 
we find 
$$
w=u+G.
$$
Here 
\begin{subequations}\label{VBL1}
\begin{align}
\label{V1}
& H=\{ v\in H^1(\Omega) : \gamma_0(v)=0 \text{ on } \partial \Omega_{D} \}, \\
\label{B1}
& B[u,v]=\int_{\Omega} (\tau D_i(x)e^{-\psi^n_i} \nabla u\cdot \nabla v+e^{-\psi^n_i} u v)dx,\\
\label{L1}
& Lv=\int_{\Omega} (\rho^n_i-e^{-\psi^n_i}G)v - \tau D_i(x)e^{-\psi^n_i}\nabla G\cdot \nabla vdx.
\end{align}
\end{subequations}
Under the assumptions, the celebrated Lax-Milgram theorem (\cite{GT97} Theorem 5.8) ensures that the variational problem (\ref{BL}) admits a unique solution $u\in H$. We thus obtain 
$$
 \rho^{n+1}_i =e^{-\psi^n_i} (u+G).
$$ 
Regularity for $\rho_i^{n+1}$ follows from the classical elliptic regularity for $u$.  

Similarly,  the variational problem for (\ref{EE2}) can also be written as (\ref{BL})
with 
\begin{subequations}\label{VBL1}
\begin{align}
\label{B2}
& B[u,v]=\int_{\Omega} \epsilon(x) \nabla u \cdot \nabla v dx,\\
\label{L2}
& Lv=\int_{\Omega} 4\pi\left( f(x)+\sum_{i=1}^mq_i\rho^{n+1}_i\right)v - 
\epsilon(x) \nabla G\cdot \nabla v dx,
\end{align}
\end{subequations}
where the Dirichlet lift $G \in H^2(\Omega)$ with  $\gamma_0(G)=\phi^b(x,t_{n+1}) \text{ on } x\in \partial \Omega_D.$ 
 Here one can use the Poincar\'{e}-Friedrichs' inequality of form 
 $\|u\|_{L^2} \leq C_F \|\nabla u\|_{L^2}$, which holds if $u=0$ on a set of 
 $\partial \Omega$ with non-vanishing measure, to regain coercivity of $B$ on $H$.  The variational problem is thus well-posed, and we obtain
$$
\phi^{n+1}=u+G. 
$$
Regularity for $\phi^{n+1}$ follows from the classical elliptic regularity for $u$ and regularity for $\rho^{n+1}$. 

If $\partial \Omega_D = \emptyset$,  then $B[u, 1]=0$ requires the compatibility condition for the source 
$$
\int_{\Omega} \left( f(x)+\sum_{i=1}^m q_i\rho^{n+1}_i\right)dx=0.
$$
Due to conservation of mass,  this can be ensured by 
$$
\int_{\Omega} \left( f(x)+\sum_{i=1}^m q_i\rho^{\rm in}_i\right)dx=0.
$$
With such compatibility condition  the solution of this variational formulation exists but is not unique. In such case one can replace $H$ by 
$$H_*=\left\{v\in H^1, \quad \int_{\Omega} v dx=0 \right\},
$$ 
then by the Poincar\'{e}-Wirtinger  inequality, $B$ is actually $H_*$-
coercive. The new variational problem hence admits a unique solution and is well-posed.    

{
Finally we prove positivity of $\rho^{n+1}$ if $\rho^n \geq 0$.  Since $w=\rho_i^{n+1}e^{\psi^n_i} \in C(\bar \Omega)\cap C^2(\Omega)$, we let $x^*=\argmin_{x\in \bar \Omega}w(x)$, and distinct three cases: 

(i) If $x^*\in \partial \Omega_D$, then 
$$
w(x)\geq w(x^*)=\rho^b_i(x^*,t_{n+1})e^{\psi^b_i(x^*,t_n)} \geq 0, \quad x\in \bar \Omega.
$$

(ii) If $x^*\in \Omega$, then we can show that 
$$
w(x) \geq w(x^*)\geq \rho^n_i(x^*)e^{\psi^n_i(x^*)}\geq 0, \quad x\in \bar \Omega.
$$
In fact, from (\ref{EL1}) it follows 
\begin{align*}
\rho^n_i(x)=& e^{-\psi^n_i(x)}w(x)-\tau\nabla \cdot (D_i(x)e^{-\psi^n_i(x)}\nabla w(x))\\
=& e^{-\psi^n_i(x)}w(x)-\tau\nabla (D_i(x)e^{-\psi^n_i(x)}) \cdot \nabla w(x)-\tau D_i(x)e^{-\psi^n_i(x)}  \Delta w(x).
\end{align*}
This when evaluated at $x^*$, using  $\nabla w(x^*)=0$ and $\Delta w(x^*) \geq 0$, gives  
$$
\rho^n_i(x^*) \leq e^{-\psi^n_i(x^*)}w(x^*). 
$$
(iii) For $x^*\in \partial\Omega_N$. If $w(x^*)\geq 0$, the proof is complete. We proceed with the case that  
$$
w(x^*)<0,\quad  x^*\in \partial\Omega_N.
$$
This is possible by the Hopf strong minimum principle. 

Define the differential operator 
$$
L\xi :=\tau D_i(x)e^{-\psi^n_i(x)}\Delta \xi +\tau\nabla (D_i(x)e^{-\psi^n_i(x)})\cdot \nabla \xi -e^{-\psi^n_i(x)} \xi.
$$
We then have $Lw=-\rho_i^n(x) \leq 0$, and $w(x) \geq w(x^*)$ for all $x\in \Omega$. These together with  $w(x^*)<0$ allow us to apply Theorem 8 in \cite{PW67} to conclude 
 $(\nabla w(x^*))\cdot \mathbf{n}<0$. This is a contradiction. \\
Collecting all three cases, we have $w(x)\geq 0$ for all $x\in \bar\Omega$. 
}
\end{proof}

\section{Proof of Theorem \ref{First-Energy}.}
\begin{proof}
(i) For fixed $i$ we sum (\ref{fully1C}) over all cells to get
$$
\sum_{\alpha\in \mathcal{A}} |K_{\alpha}| (\rho^{n+1}_{i,\alpha} -\rho^n_{i,\alpha} ) = \tau \sum_{j=1}^d  \sum_{\alpha\in \mathcal{A}}  \frac{|K_{\alpha}|}{h_j}(C_{i,\alpha+e_j/2}-C_{i,\alpha -e_j/2})=0,
$$
where we used summation by parts and $C_{i, \alpha+e_j/2}=0$ for $x_{\alpha+e_j/2}\in \partial \Omega$.\\
(ii) Set 
$$S^{n}_{\alpha}=f_{\alpha}+\sum_{i=1}^m q_i\rho_{i,\alpha}^n$$ 
and 
$$
\psi^*_{i, \alpha}=\log \rho^{n+1}_{i, \alpha}+\frac{1}{k_{B}T}q_i\phi^n_{\alpha}+\frac{1}{k_BT}\mu_{i,\alpha}.
$$  
Using (\ref{energy2}) we find that 
\begin{equation}\label{E1}
\begin{aligned}
E^{n+1}_h-E^n_h=&   \sum_{\alpha \in \mathcal{A}} \sum_{i=1}^m |K_{\alpha}|\left( (\rho^{n+1}_{i,\alpha}-\rho^n_{i,\alpha})\psi^*_{i, \alpha} + \rho^n_{i,\alpha}\log \frac{\rho^{n+1}_{i,\alpha}}{\rho^n_{\alpha}}\right)\\
& + \frac{1}{k_{B}T} \sum_{\alpha \in \mathcal{A}} |K_{\alpha}| \left(\frac{1}{2}S^{n+1}_{\alpha}\phi^{n+1}_{\alpha}+ \frac{1}{2}S^{n}_{\alpha}\phi^{n}_{\alpha} - S^{n+1}_{\alpha}\phi^{n}_{\alpha}\right).
\end{aligned}
\end{equation}
Using $\log X\leq  X -1$ for $X>0$ and the mass conservation,  
 we have  
$$
\sum_{\alpha \in \mathcal{A}}  |K_{\alpha}| \rho^n_{i,\alpha}\log \frac{\rho^{n+1}_{i,\alpha}}{\rho^n_{\alpha}}\leq 
\sum_{\alpha \in \mathcal{A}}  |K_{\alpha}|(\rho^{n+1}_{i,\alpha}-\rho^n_{\alpha})=0.
$$
Also one can verify that 
 $$
 \sum_{\alpha\in \mathcal{A}} |K_{\alpha}|S^{n+1}_{\alpha}\phi^n_{\alpha}=\sum_{\alpha\in \mathcal{A}} |K_{\alpha}|S^{n}_{\alpha}\phi^{n+1}_{\alpha},$$
with which we obtain 
$$
\sum_{\alpha \in \mathcal{A}} |K_{\alpha}| \left(\frac{1}{2}S^{n+1}_{\alpha}\phi^{n+1}_{\alpha}+ \frac{1}{2}S^{n}_{\alpha}\phi^{n}_{\alpha} - S^{n+1}_{\alpha}\phi^{n}_{\alpha}\right)=
\frac{1}{2}\sum_{\alpha \in \mathcal{A}} |K_{\alpha}| (S^{n+1}_{\alpha}-S^n_{\alpha})(\phi^{n+1}_{\alpha}-\phi^{n}_{\alpha}).
$$
Insertion of these into (\ref{E1}) gives
\begin{equation}\label{E11}
E^{n+1}_h-E^n_h \leq -\tau I^n+\tau^2 I\!I^n, 
\end{equation}
where 
\begin{align*}
I^n & =- \sum_{\alpha \in \mathcal{A}} \sum_{i=1}^m |K_{\alpha}| \left(\frac{ \rho^{n+1}_{i,\alpha}-\rho^n_{i,\alpha}}{\tau} \right) \psi^*_{i, \alpha},   \\
I\!I^n & = \frac{1}{2k_{B}T} \sum_{\alpha \in \mathcal{A}} |K_{\alpha}| \left(\frac{S^{n+1}_{\alpha}-S^n_{\alpha}}{\tau}\right)\left(\frac{\phi^{n+1}_{\alpha}-\phi^{n}_{\alpha}}{\tau}\right).
\end{align*}
By using (\ref{fully1C}) and summation by parts, we have 
\begin{equation}\label{bd11}
\begin{aligned}
I^n=&  -\sum_{i=1}^m \sum_{\alpha \in \mathcal{A}} \sum_{j=1}^d |K_{\alpha}| \left( \frac{C_{i,\alpha+e_j/2}-C_{i,\alpha-e_j/2}}{h_j}\right ) \psi^*_{i,\alpha}\\
=&  \sum_{i=1}^m\sum_{j=1}^d \sum_{\alpha(j) \neq N_j}  \frac{|K_{\alpha}|}{h_j}  \left(\psi^*_{i,\alpha+e_j}- \psi^*_{i,\alpha}
\right)C_{i,\alpha+e_j/2}.
\end{aligned}
\end{equation}
Note that 
$$
C_{i,\alpha+e_j/2} =\frac{ D_i(x_{ \alpha+e_j/2}) e^{-\psi^n_{i,\alpha+e_j/2}} }{h_j}\left(e^{\psi^*_{i,\alpha+e_j}}- e^{\psi^*_{i,\alpha}}
\right),
$$ 
hence $I^n\geq 0$. 

We pause to discuss the special case with $I^n=0$. In such case we must have 
$\psi^*_{i,\alpha+e_j}=\psi^*_{i,\alpha}$ for each $i, j$ and $\alpha\in \mathcal{A}$, 
which implies $C_{i,\alpha+e_j/2}=0$ for each $i$,  $j$ and $\alpha\in \mathcal{A}$. 
Thus, we have $\rho^{n+1}_{i,\alpha}=\rho^n_{i,\alpha}$,  hence 
$$
S^{n}_{\alpha}=f_{\alpha}+\sum_{i=1}^m q_i\rho_{i,\alpha}^n=f_{\alpha}+\sum_{i=1}^m q_i\rho_{i,\alpha}^{n+1}=S^{n+1}_{\alpha}, \quad \forall \alpha\in \mathcal{A},
$$
therefore $I\!I^n=0$ and $E^{n+1}_h-E^n_h\leq0 $, this is (\ref{En}) with $I^n=0$. 

From now on we only consider  the case $I^n>0$. We proceed to estimate $I\!I^n$,
\begin{equation}\label{bd21}
\begin{aligned}
I\!I^n=&  \frac{1}{2k_{B}T} \sum_{\alpha \in \mathcal{A}} |K_{\alpha}| \left(\frac{S^{n+1}_{\alpha}-S^n_{\alpha}}{\tau}\right) \left( \frac{\phi^{n+1}_{\alpha}-\phi^{n}_{\alpha}}{\tau}\right)\\
=&  -\frac{1}{8\pi k_{B}T} \sum_{\alpha \in \mathcal{A}} \sum_{j=1}^d \frac{|K_{\alpha}|}{\tau^2h_j}(\Phi_{i,\alpha+e_j/2}^{n+1}-\Phi_{i,\alpha-e_j/2}^{n+1}-
\Phi_{i,\alpha+e_j/2}^{n}+\Phi_{i,\alpha-e_j/2}^{n})(\phi^{n+1}_{\alpha}-\phi^{n}_{\alpha})\\
=&  \frac{1}{8\pi k_{B}T}  \sum_{j=1}^d \sum_{\alpha(j)\ne N_j} \frac{|K_{\alpha}| }{\tau^2h_j}(\Phi_{i,\alpha+e_j/2}^{n+1}-\Phi_{i,\alpha+e_j/2}^n) (\phi^{n+1}_{\alpha+e_j}-\phi^{n}_{\alpha+e_j}-\phi^{n+1}_{\alpha}+\phi^{n}_{\alpha})\\
=&  \frac{1}{8\pi k_{B}T}  \sum_{j=1}^d \sum_{\alpha(j)\ne N_j} |K_{\alpha}| \frac{\epsilon_{\alpha+e_j/2}}{\tau^2h^2_j} (\phi^{n+1}_{\alpha+e_j}-\phi^{n}_{\alpha+e_j}-\phi^{n+1}_{\alpha}+\phi^{n}_{\alpha})^2\geq 0.
\end{aligned}
\end{equation}
Here the second  equality is obtained by using the equation (\ref{fully1Ps}), the last equality is obtained by using the definition (\ref{Ps2fully1}) of $\Phi^n_{i,\alpha+e_j/2}$.

From (\ref{bd11}) and (\ref{bd21}), we see that the energy dissipation inequality (\ref{En}) is satisfied if 
\begin{equation}\label{reason}
\tau\leq \tau^* \leq \frac{I^n}{2 I\!I^n}.
\end{equation}
In the remaining of the proof we will quantify $\tau^*$ from estimating the lower bound of $\frac{I^n}{2 I\!I^n}.$

Subtracting (\ref{fully1Ps}) at time level $t=t_{n+1}$ and $t=t_n$, one has
\begin{equation}\label{bd22}
 -\sum_{j=1}^d \frac{\Phi^{n+1}_{\alpha+e_j/2}-\Phi^{n+1}_{\alpha-e_j/2}-\Phi^n_{\alpha+e_j/2}+\Phi^n_{\alpha-e_j/2}}{h_j}=4\pi \sum_{i=1}^{m}q_i(\rho^{n+1}_{i,\alpha}-\rho^n_{i,\alpha}),
\end{equation}
multiplying by $|K_{\alpha}|(\phi^{n+1}_{\alpha}-\phi^n_{\alpha})$ and summing over $\alpha\in \mathcal{A}$ leads to 
\begin{equation}\label{bd23}
\begin{aligned}
& -\sum_{j=1}^d\sum_{\alpha\in \mathcal{A}} \frac{K_{\alpha}}{h_j} (\Phi^{n+1}_{\alpha+e_j/2}-\Phi^{n+1}_{\alpha-e_j/2}-\Phi^n_{\alpha+e_j/2}+\Phi^n_{\alpha-e_j/2})(\phi^{n+1}_{\alpha}-\phi^n_{\alpha})\\
 & =4\pi \sum_{i=1}^{m}\sum_{\alpha\in \mathcal{A}} q_i |K_{\alpha}| (\rho^{n+1}_{i,\alpha}-\rho^n_{i,\alpha})(\phi^{n+1}_{\alpha}-\phi^n_{\alpha}).
 \end{aligned}
\end{equation}
Similar to (\ref{bd21}), the left hand side of (\ref{bd23}) reduces to
\begin{equation}\label{bd24}
LHS=\sum_{j=1}^d \sum_{\alpha(j)\ne N_j} |K_{\alpha}| \frac{\epsilon_{\alpha+e_j/2}}{h^2_j} (\phi^{n+1}_{\alpha+e_j}-\phi^{n}_{\alpha+e_j}-\phi^{n+1}_{\alpha}+\phi^{n}_{\alpha})^2.
\end{equation}
We estimate the right hand side of (\ref{bd23}) by using the equation (\ref{fully1C}):
\begin{equation}\label{bd25}
\begin{aligned}
RHS=& 4\pi \sum_{i=1}^{m}\sum_{\alpha\in \mathcal{A}} q_i|K_{\alpha}| (\rho^{n+1}_{i,\alpha}-\rho^n_{i,\alpha})(\phi^{n+1}_{\alpha}-\phi^n_{\alpha})\\
=&4\pi \tau \sum_{i=1}^{m}\sum_{\alpha\in \mathcal{A}}\sum_{j=1}^d q_i |K_{\alpha}|\frac{1}{h_j}(C_{i,\alpha+e_j/2}-C_{i,\alpha-e_j/2})(\phi^{n+1}_{\alpha}-\phi^n_{\alpha})\\
=& -4\pi \tau \sum_{i=1}^{m} \sum_{j=1}^d \sum_{\alpha(j)\ne N_j} q_i |K_{\alpha}|\frac{1}{h_j} C_{i,\alpha+e_j/2} ( \phi^{n+1}_{\alpha+e_j}-\phi^{n}_{\alpha+e_j}-\phi^{n+1}_{\alpha}+\phi^{n}_{\alpha} ).
\end{aligned}
\end{equation}
Note that
$$
LHS\geq \epsilon_{min}\sum_{j=1}^d \sum_{\alpha(j)\ne N_j}  |K_{\alpha} | \left(\frac{\phi^{n+1}_{\alpha+e_j}-\phi^{n}_{\alpha+e_j}-\phi^{n+1}_{\alpha}+\phi^{n}_{\alpha}}{h_j}\right)^2.
$$
Using the Cauchy-Schwarz inequality we see that
\begin{small}
\begin{equation}\label{bd26}
\begin{aligned}
RHS & \leq  4\pi \tau  \sum_{i=1}^{m} |q_i | \left( \sum_{j=1}^d \sum_{\alpha(j)\ne N_j} |K_{\alpha}| C_{i,\alpha+e_j/2} \left(\frac{ \phi^{n+1}_{\alpha+e_j}-\phi^{n}_{\alpha+e_j}-\phi^{n+1}_{\alpha}+\phi^{n}_{\alpha}}{h_j} \right)\right)\\
& \leq 4\pi \tau  \sum_{i=1}^{m} |q_i |  \left[ \sum_{j=1}^d \sum_{\alpha(j)\ne N_j} |K_{\alpha}|C^2_{i,\alpha+e_j/2} \right]^{1/2} \\
& \qquad  \times \left [  \sum_{j=1}^d \sum_{\alpha(j)\ne N_j} |K_{\alpha} |\left( \frac{\phi^{n+1}_{\alpha+e_j}-\phi^{n}_{\alpha+e_j}-\phi^{n+1}_{\alpha}+\phi^{n}_{\alpha} }{h_j}\right)^2\right]^{1/2}.
\end{aligned}
\end{equation}
\end{small}
We thus obtain
\begin{equation}\label{bd27}
\begin{aligned}
& \sum_{j=1}^d \sum_{\alpha(j)\ne N_j} |K_{\alpha}| \left(\frac{\phi^{n+1}_{\alpha+e_j}-\phi^{n}_{\alpha+e_j}-\phi^{n+1}_{\alpha}+\phi^{n}_{\alpha}}{h_j}\right)^2 \\
&\frac{ \leq 16\pi^2 \tau^2 }{\epsilon^2_{min}} \left[ \sum_{i=1}^{m} |q_i |  \left( \sum_{j=1}^d \sum_{\alpha(j)\ne N_j} |K_{\alpha}| C^2_{i,\alpha+e_j/2} \right) ^{1/2} \right]^2\\
& \leq \frac{16\pi^2 \tau^2 }{\epsilon^2_{min}} \left(\sum_{i=1}^{m} q^2_i  \right)  \sum_{i=1}^{m}    \sum_{j=1}^d \sum_{\alpha(j)\ne N_j}|K_{\alpha}| C^2_{i,\alpha+e_j/2}. \\
\end{aligned}
\end{equation}
Upon insertion into (\ref{bd21})
\begin{equation}\label{bd28}
I\!I^n\leq C  \sum_{i=1}^{m}    \sum_{j=1}^d \sum_{\alpha(j)\ne N_j} |K_{\alpha}| C^2_{i,\alpha+e_j/2},
\end{equation}
where $C= \frac{2 \epsilon_{max} \pi   \sum_{i=1}^m q_i^2}{\epsilon^2_{min} k_{B}T}$. 
We use ({\ref{bd11}}) and (\ref{bd28}) to obtain:
\begin{equation}\label{bd29}
\begin{aligned}
\frac{I^n}{2I\!I^n} \geq & \frac{ \sum_{i=1}^m\sum_{j=1}^d \sum_{\alpha(j) \neq N_j}  \frac{|K_{\alpha}|}{h_j}C_{i,\alpha+e_j/2} (\psi^*_{i,\alpha+e_j}- \psi^*_{i,\alpha})}{ 2C \sum_{i=1}^{m}    \sum_{j=1}^d \sum_{\alpha(j)\ne N_j} |K_{\alpha}|C^2_{i,\alpha+e_j/2}}\\
\geq & \frac{1}{2C}\min_{i,j,\alpha} \left\{ \frac{\psi^*_{i,\alpha+e_j}- \psi^*_{i,\alpha}}{h_j C_{i,\alpha+e_j/2}}\right\}\\
= & \frac{1}{2C}\min_{i,j,\alpha} \left\{ \frac{\psi^*_{i,\alpha+e_j}- \psi^*_{i,\alpha}}{D_{i,\alpha+e_j/2} e^{-\psi^n_{i,\alpha+e_j/2}} (e^{\psi^*_{i,\alpha+e_j}}-e^{\psi^*_{i,\alpha}})} \right\} \quad  \text{ by the mean-value theorem}\\
=&  \frac{1}{2C}\min_{i,j,\alpha}\left\{ \frac{1}{D_{i,\alpha+e_j/2} e^{-\psi^n_{i,\alpha+e_j/2}} e^{(\theta\psi^*_{i,\alpha+e_j}+(1-\theta)\psi^*_{i,\alpha})} }\right\},
\end{aligned}
\end{equation}
where $\theta\in (0, 1).$ By using the harmonic mean for $e^{-\psi^n_{i,\alpha+e_j/2}}$, we have
\begin{align*}
\frac{1}{ e^{-\psi^n_{i,\alpha+e_j/2}} e^{(\theta\psi^*_{i,\alpha+e_j}+(1-\theta)\psi^*_{i,\alpha})} }=& \frac{ e^{((\theta-1)\psi^n_{i,\alpha}-\theta\psi^n_{i,\alpha+e_j}  )}    }{ (\rho^{n+1}_{i,\alpha+e_j} )^{\theta}(\rho^{n+1}_{i,\alpha})^{1-\theta}         }    \cdot \frac{2e^{\psi^n_{i,\alpha+e_j}+\psi^n_{i,\alpha}}}{ e^{\psi^{n}_{i,\alpha+e_j}}+e^{\psi^n_{i,\alpha}}} \\
=&  \frac{1  }{ (\rho^{n+1}_{i,\alpha+e_j} )^{\theta}(\rho^{n+1}_{i,\alpha})^{1-\theta}        }\cdot  \frac{2e^{(1-\theta)\psi^n_{i,\alpha+e_j}+\theta\psi^n_{i,\alpha}}}{ e^{\psi^{n}_{i,\alpha+e_j}}+e^{\psi^n_{i,\alpha}}}      \\
\geq &  \frac{ 2e^{\min\left \{ \psi^n_{i,\alpha+e_j}, \psi^n_{i,\alpha} \right \}} } { 2M    e^{ \max \left \{ \psi^{n}_{i,\alpha+e_j},  {\psi^n_{i,\alpha}} \right \} }  }  \\
=&  \frac{ e^{- | \psi^n_{i,\alpha+e_j}- \psi^n_{i,\alpha} |} } { M     } ,
\end{align*}
where $M=\max_{i,\alpha, n} \rho^{n}_{i,\alpha}$, thus 
\begin{equation}\label{ll}
\frac{I^n}{2 I\!I^n} \geq  \frac{1} {2C D_{max}M } e^{-\max_{i,j,\alpha}|\psi^n_{i,\alpha+e_j} -\psi^n_{i,\alpha}|}.
\end{equation}
For geometric mean or algebraic mean when  used for the evaluation of $e^{-\psi^n_{i,\alpha+e_j/2}}$ we can verify either the same or bigger bound than the right hand side of in (\ref{ll}).
\end{proof}


\end{document}